\documentclass[12pt]{article}
\usepackage[top=2.54cm, bottom=2.54cm, left=2.5cm, right=2.5cm]{geometry}
\usepackage[tbtags]{amsmath}
\usepackage{amssymb}
\usepackage{amsthm}
\usepackage{fancyhdr}
\usepackage{latexsym}
\usepackage{mathrsfs}
\usepackage{xcolor}
\usepackage{wasysym}
\usepackage{fancyhdr}
\usepackage{float}
\usepackage{graphicx}
\usepackage[numbers,sort&compress]{natbib}
\usepackage{pict2e,color}
\usepackage{tikz}
\usepackage{tkz-graph}
\usetikzlibrary{arrows,shapes,chains}
\usepackage{enumerate}
\usepackage {verbatim}
\usepackage{tikz-cd}
\usepackage{rotating}
\usepackage{subcaption}
\allowdisplaybreaks[4]

\newtheorem{theorem}{\bf Theorem}[section]
\newtheorem{lemma}[theorem]{Lemma}

\newtheorem{conjecture}[theorem]{Conjecture}

\newtheorem{definition}[theorem]{Definition}
\newtheorem{proposition}[theorem]{Proposition}
\theoremstyle{remark}

\newcommand{\norm}[1]{\left\lVert#1\right\rVert}

\makeatletter
\newcommand*{\rom}[1]{\expandafter\@slowromancap\romannumeral #1@}
\makeatother

\begin{document}
	\title{Finding irregular subgraphs via local adjustments}
	
	\author{Jie Ma\thanks{School of Mathematical Sciences, University of Science and Technology of China, Hefei, Anhui, 230026, China. Research supported by National Key Research and Development Program of China 2023YFA1010201 and National Natural Science Foundation of China grant 12125106. Email: jiema@ustc.edu.cn.}~~~
		Shengjie Xie\thanks{School of Mathematical Sciences, University of Science and Technology of China, Hefei, Anhui, 230026, China. Email: jeff\_776532@mail.ustc.edu.cn.}
	}
	

\date{}
\maketitle

\begin{abstract}
	For a graph $H$, let $m(H,k)$ denote the number of vertices of degree $k$ in $H$. 
	A conjecture of Alon and Wei states that for any $d\geq 3$,
	every $n$-vertex $d$-regular graph contains a spanning subgraph $H$ satisfying $|m(H,k)-\frac{n}{d+1}|\leq 2$ for every $0\leq k \leq d$.
	This holds easily when $d\leq 2$.
	An asymptotic version of this conjecture was initially established by Frieze, Gould, Karo\'nski and Pfender, subsequently improved by Alon and Wei, and most recently enhanced by Fox, Luo and Pham, approaching its complete range. 
	All of these approaches relied on probabilistic methods.

	In this paper, we provide a novel framework to study this conjecture, based on localized deterministic techniques which we call local adjustments.  
	We prove two main results. 
	Firstly, we show that every $n$-vertex $d$-regular graph contains a spanning subgraph $H$ satisfying $|m(H,k)-\frac{n}{d+1}|\leq 2d^2$ for all $0\leq k \leq d$, which provides the first bound independent of the value of $n$.
	Secondly, we confirm the case $d=3$ of the Alon-Wei Conjecture in a strong form. 
	Both results can be generalized to multigraphs and yield efficient algorithms for finding the desired subgraphs $H$. 
	Furthermore, we explore a generalization of the Alon-Wei Conjecture for multigraphs and its connection to the Faudree-Lehel Conjecture concerning irregularity strength. 
\end{abstract}

\section{Introduction}
In this paper, we use the convention that the term ``multigraph'' refers to a graph allowing parallel edges but with no loops, while the term ``graph'' denotes a {\it simple} graph (i.e., a graph without parallel edges and loops).
For a multigraph $H$ and an integer $k\geq 0$, 
let $m(H,k)$ denote the number of vertices of degree $k$ in $H$, 
and $m(H)=\max_{k\geq 0} m(H,k).$

Alon and Wei \cite{ALO2022} proposed the problem of finding highly irregular subgraphs $H$ in an $n$-vertex $d$-regular graph $G$.
In this context, $m(H)$ can be viewed as a measure of the irregularity of subgraphs $H$, where a smaller value of $m(H)$ indicates a more dispersed distribution of the vertex degrees of $H$, thus suggesting a higher level of irregularity.
Additionally, using the pigeonhole principle, it is easy to see that any spanning subgraph $H$ of $G$ satisfies $m(H)\geq \lceil\frac{n}{d+1}\rceil$. 
Alon and Wei \cite{ALO2022} made the following conjecture, 
which asserts the existence of a spanning subgraph $H$ with a small value of $m(H)$ that nearly matches this lower bound.

\begin{conjecture}[Alon-Wei \cite{ALO2022}, see Conjecture 1.1]\label{conj}
	Every $d$-regular graph $G$ on $n$ vertices contains a spanning subgraph $H$ such that $$\left|m(H,k) - \frac{n}{d+1}\right|\leq 2 \mbox{ holds for for all } 0\leq k \leq d.$$
\end{conjecture}

\noindent It is easy to see that this conjecture holds when $d\leq 2$.
If true, this is optimal in the sense that the bound cannot be improved by ``$\leq 1$''. A specific example is the vertex disjoint union of two four-cycles, as shown in \cite{ALO2022}. 
Additionally, other examples that support this conjecture include all complete bipartite graphs $K_{d,d}$ for odd integers $d$.\footnote{For odd $d$, it can be demonstrated (through non-trivial analysis) that any spanning subgraph $H$ of $K_{d,d}$ has some $0\leq k\leq d$ such that either $m(H,k)=0$ or $m(H,k)\geq 3$.}

Conjecture~\ref{conj} is closely related to the problem of determining the {\it irregularity strength} $s(G)$ of a graph $G$.
The irregularity strength $s(G)$, introduced in \cite{CJLORS1988}, denotes the least positive integer $s$
for which there exists a function $f:E(G)\to \{1,2,\ldots, s\}$ such that 
the sums of $f(e)$ over all incident edges $e$ to each vertex of $G$ are distinct. 
A well-known conjecture of Faudree and Lehel \cite{FL1988} states that 
there exists an absolute constant $c>0$ such that for any $d\geq 2$, 
every $n$-vertex $d$-regular graph $G$ has $s(G)\leq \frac{n}{d}+c$.
This conjecture is tight up to the constant $c$ and remains open, despite recent breakthroughs in \cite{Prz2022, PW2023-1, PW2023-2}. 
Alon and Wei \cite{ALO2022} built a connection between these two problems (see its Theorems 1.7 and 1.8). 
In particular, they proved that any $d$-regular graph $G$ contains a spanning subgraph $H$ with $m(H)\leq 2s(G)-2$.
For further exploration of the irregularity strength, 
we recommend referring to the literature that includes works such as
\cite{Amar1993, BK2004, CL2008, DGG1992, EHLW1990, FJLR1989, FL1988, FGKP2010, FGKP2002, Gya1998, KKP2011, Leh1991, MP2014, Nie2000, Prz2008, Prz2009, Prz2022, PW2023-1, PW2023-2}.

Returning to Conjecture~\ref{conj}, 
an early result of Frieze, Gould, Karo\'{n}ski and Pfender \cite{FGKP2002} shows that every $n$-vertex $d$-regular graph contains a spanning subgraph $H$ with
\begin{equation}\label{equ:1+o(1)}
	m(H,k)=(1+o(1))\frac{n}{d+1} \mbox{ for all $0\leq k\leq d$}
\end{equation}
whenever $d\leq (n/\log n)^{1/4}$ is provided.
Alon and Wei \cite{ALO2022} enhanced this result by demonstrating the same statement under a relaxed condition of $d=o((n/\log n)^{1/3})$. 
Very recently,  Fox, Luo and Pham \cite{FLP24} proved that \eqref{equ:1+o(1)} holds under an even weaker condition, namely $d=o(n/(\log n)^{12})$.
It is worth noting that all of these aforementioned results are obtained by employing probabilistic methods, specifically by considering random spanning subgraphs.

In this paper, we present a novel deterministic approach for addressing Conjecture~\ref{conj}.
Our approach is centered around the recursive updating of spanning subgraphs $H$ through a series of localized operations, referred to as {\it local adjustments}. 
In comparison to the previous probabilistic approach, our method demonstrates superior performance specifically when applied to {\it sparse} $d$-regular graphs, offering two distinct advantages. 
Firstly, it provides a deterministic algorithm with efficient time complexity for the construction of irregular spanning subgraphs. 
Secondly, it extends its applicability to a wider range of graphs, including all $d$-regular {\it multigraphs}.
This extension is particularly significant as it plays a critical role in our discussion of a generalization of Conjecture~\ref{conj} that would imply the Faudree-Lehel conjecture (see Section~\ref{Sec:remarks} for details).

Our first result is the following. 
Throughout this paper, when we refer to a subgraph of a multigraph, it is implied that the subgraph is also a multigraph.

\begin{theorem}\label{thm1}
	Every $d$-regular multigraph $G$ on $n$ vertices contains a spanning subgraph $H$ such that  $$\left|m(H,k) - \frac{n}{d+1}\right|\leq 2d^2
	\mbox{ holds for all $0\leq k \leq d$}.$$
	Moreover, there exists an algorithm with time complexity $O_d(n^{d+1})$ to find such $H$.
\end{theorem}

To determine the leading coefficient hidden in the time complexity $O_d(n^{d+1})$, it suffices to choose $O(d^{d+2})$.
We point out that this result implies the first bound for Conjecture~\ref{conj} that is independent of the value of $n$.
Furthermore, when compared with the result of Fox, Luo and Pham \cite{FLP24}, Theorem~\ref{thm1} gives a more precise bound,
particularly when $d = o(n^{1/3})$.

Our second result focuses on cubic multigraphs and confirms the validity of Conjecture~\ref{conj} for the case $d=3$ (in fact, in a slightly stronger form, as stated in Theorem~\ref{thm2:stronger}).

\begin{theorem}\label{thm2}
	Every cubic multigraph $G$ on $n$ vertices contains a spanning subgraph $H$ such that
	$$\left|m(H,k) - \frac{n}{4}\right|\leq 2 \mbox{ holds for all $0\leq k \leq 3$}. $$
	Moreover, there exists a linear time algorithm to find such a spanning subgraph $H$.
\end{theorem}

For an overview of both results, we refer readers to the beginning of Section~\ref{Sec:thm1} for an outline of Theorem~\ref{thm1}, and Subsection \ref{subsec:char-cubic} for an outline of Theorem~\ref{thm2}.
We add a side remark here that while both proofs of Theorems~\ref{thm1} and \ref{thm2} utilize local adjustments, the methodologies differ in detail.
The proof of Theorem~\ref{thm1} involves continuously updating spanning subgraphs of the host multigraph $G$.
On the other hand, the proof of Theorem~\ref{thm2} incorporates an inductive argument 
where a sequence of host multigraphs $G_i$ is constructed, and the spanning subgraphs $H_i$ are updated within each $G_i$.

The remaining sections of the paper are organized as follows.
In Section~\ref{Sec:prim},  we introduce the necessary notation and concepts, including an explanation of the mean of local adjustments.
In Section~\ref{Sec:thm1}, we present the completed proof of Theorem~\ref{thm1}.
In Section~\ref{Sec:thm2}, we provide the proof of Theorem~\ref{thm2}.
Finally, in Section~\ref{Sec:remarks}, we conclude with some closing remarks.
Throughout this paper, let $[k]$ denote the set $\{1,2,...,k\}$ for positive integers $k$.

\section{Irregularity vectors and local adjustments}\label{Sec:prim}
In this section, we provide notation and concepts necessary for the upcoming proofs.
One of the key notions we define is a quantitative measure that captures the irregularity exhibited by spanning subgraphs (see the definition of \eqref{equ:vector} for the {\it irregularity vector}). Additionally, we introduce our {\it local adjustment} techniques\footnote{By a local adjustment, we mean an operation of adding or deleting edges of some specific multigraphs (for further details, we refer to Subsections~\ref{subsec:local-adj} and \ref{subsec:char-cubic}).} and illustrate how these techniques contribute to the irregularity measure.

\subsection{Irregularity vectors}
Throughout this section, let $G$ be a fixed $n$-vertex $d$-regular multigraph and $H$ be a spanning subgraph of $G$. For $0\leq i\leq d$, let $V_i^H = \{v\in V(G): d_H(v) = i\}$ be the set of vertices with degree $i$ in $H$.
So $m(H,i)=|V_i^H|$.
By writing $a_i(H) = |V_i^H| - \frac{n}{d+1}$, we define $$\mathbf{a}(H) = (a_0(H), a_1(H), \ldots, a_d(H))$$ to be the vector whose entries indicate the difference between $|V_i^H|$ and $\frac{n}{d+1}$.
The Alon-Wei conjecture is equivalent to finding a spanning subgraph $H$ such that the $L_\infty$ norm $\norm{\mathbf{a}(H)}_{\infty}$ of $\mathbf{a}(H)$ is at most two.
However, in establishing the general bound of Theorem \ref{thm1}, our focus primarily lies on bounding the following convolution version $\mathbf{b}(H)$ of $\mathbf{a}(H)$:
Define $b_0(H) = 0$
and for $1\leq i\leq d+1$, define $b_i(H) = \sum_{j=0}^ {i-1} a_{j}(H)$.
(Note that we have $b_{d+1}(H)=0$.)
Let
\begin{equation}\label{equ:vector}
	\mathbf{b}(H) = (b_1(H), b_2(H), \ldots, b_d(H))
\end{equation}
be the {\it irregularity vector} of $H$.
It is clear that $\mathbf{a}(H)$ and $\mathbf{b}(H)$ completely determine each other and moreover, since $|a_i(H)|=|b_{i+1}(H)-b_i(H)|\leq |b_{i+1}(H)|+|b_i(H)|$, we have
\begin{equation}\label{equ:a<2b}
	\norm{\mathbf{a}(H)}_{\infty}\leq 2 \norm{\mathbf{b}(H)}_{\infty}.
\end{equation}
An advantage of the approach of bounding $\norm{\mathbf{b}(H)}_{\infty}$ is its simplicity in evaluating the changes of $\mathbf{b}(H)$ that result from specific local adjustments; see Lemma~\ref{change_b}.

The following proposition presents another useful perspective we adopt in the proof. It establishes the symmetric properties of $\mathbf{a}$ and $\mathbf{b}$ between a spanning subgraph $H$ of a $d$-regular multigraph $G$ and its {\it complement} $G\backslash H = (V(G), E(G)\backslash E(H))$ with respect to $G$.

\begin{proposition}\label{sym}
	Let $G$ be a $d$-regular multigraph and $H$ be a spanning subgraph of $G$. Then
	$a_i(G\backslash H) = a_{d-i}(H)$ for every $0\leq i\leq d$, and $b_i(G\backslash H) = - b_{d + 1 - i}(H)$ for every $1\leq i\leq d$.
\end{proposition}

\begin{proof}
	For every $0\leq i\leq d$, we have $V_i^{G\backslash H} = V_{d-i}^H$
	and thus $a_i(G\backslash H) = a_{d-i}(H)$.
	Since $\sum_{j=0}^d a_j(H)=0$,
	this implies that for every $1\leq i\leq d$,
	$$b_i(G\backslash H) = \sum\limits_{j=0}^{i-1}a_j(G\backslash H) = \sum\limits_{j=d-i+1}^d a_j(H) = - \sum_{j=0}^{d-i} a_j(H) = - b_{d-i+1}(H),$$ completing the proof.
\end{proof}

We observe from this proposition that finding a satisfactory spanning subgraph $H$ (say with $\norm{\mathbf{a}(H)}_{\infty}\leq C$) is equivalent to finding a satisfactory $G\backslash H$ (with the same bound $\norm{\mathbf{a}(G\backslash H)}_{\infty}\leq C$). 
In this context, we say that $H$ and $G\backslash H$ are {\it symmetric}.

\subsection{Local adjustments and their impacts on irregularity vectors}\label{subsec:local-adj}
Now we introduce a family of specialized multigraphs (which are called {\it multi-stars}, see Figure~\ref{Fig:multi-star}) that serve as the foundation for our local adjustment approach.
Subsequently, we explore their influence on the irregularity vector (see Lemma~\ref{change_b}).

Recall $H$ is a spanning subgraph of a multigraph $G$.
For an edge $xy\in E(G)$, we call $xy$ an \textbf{H(i,j)-edge} if $d_H(x)=i$ and $d_H(y)=j$.
We emphasize that $i$ can be $j$ or not, and $xy$ can be an edge in $H$ or $G\backslash H$.
Let $k, \ell_1, \ldots, \ell_s, \alpha_1,\ldots,\alpha_s\geq 1$ be integers.
Let {\bf $S_{(k; \ell_1, \ldots, \ell_s;\alpha_1,\ldots,\alpha_s)}(H)$} denote a multi-star of $H$ with center $u$ and distinct leaves $v_1,\ldots, v_s$ such that $d_H(u) = k$, $d_H(v_i)=\ell_i$ and there are $\alpha_i$ multi-edges between $u$ and $v_i$ for each $i\in [s]$.
Note that every edge between $u$ and $v_i$ in this multi-star is an $H(k,\ell_i)$-edge.
Similarly, let $\overline{S}_{(k; \ell_1, \ldots, \ell_s;\alpha_1,\ldots,\alpha_s)}(H)$ denote a multi-star of $G\backslash H$ with center $u$ and distinct leaves $v_1,\ldots, v_s$ such that $d_H(u) = k$, $d_H(v_i)=\ell_i$ and there are $\alpha_i$ multi-edges between $u$ and $v_i$ for each $i\in [s]$.

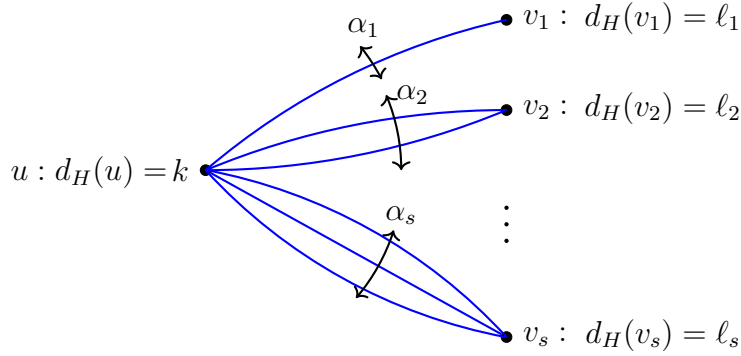
\begin{figure}[H]
	\begin{center}
		\begin{tikzpicture}[thin][node distance=1cm,on grid]
			\node[circle,inner sep=0.5mm,draw=black,fill=black](v)at(0,0) [label=left:
			$k$] {};
			\node[circle,inner sep=0.5mm,draw=black,fill=black](u1)at(4,2) [label=right:$v_1:$] {};
			\node[circle,inner sep=0.5mm,draw=black,fill=black](u2)at(4,0.8) [label=right:$v_2:$] {};
			\node[circle,inner sep=0.5mm,draw=black,fill=black](us)at(4,-2.22) [label=right:$v_s:$] {};
			\node[circle,inner sep=0.15mm,draw=black,fill=black](i1)at(4,-0.5) {};
			\node[circle,inner sep=0.15mm,draw=black,fill=black](i2)at(4,-0.72){};
			\node[circle,inner sep=0.15mm,draw=black,fill=black](i3)at(4,-0.94) {};
			
			\draw[thick,color=blue] (4,2) arc(102.53:130.60:9.220);
			\draw[thick,color=blue] (4,0.8) arc(90:112.62:10.4);
			\draw[thick,color=blue] (0,0) arc(270:292.62:10.4);
			\draw[thick,color=blue] (4,-2.22) arc(42.51:79.38:7.236);
			\draw[thick,color=blue] (0,0) arc(222.51:259.38:7.236);
			\draw[thick,color = blue](v)to(us);
			\draw[thick,<->,black] (2.34,1.21) arc(27.31:38.73:2.637);
			\draw(2.1,1.9)node {$\alpha_1$};
			\draw[thick,<->,black] (2.6,0) arc(0:22.62:2.6);
			\draw(2.75,1)node {$\alpha_2$};
			\draw[thick,<->,black] (2,-1.7) arc(319.64:342.26:2.625);
			\draw(2.6,-0.6)node {$\alpha_s$};
			
			\draw(-1.55,-0.03)node {$u:d_H(u)=$};
			\draw(6.08,2.03)node {$d_H(v_1)=\ell_1$};
			\draw(6.08,0.83)node {$d_H(v_2)=\ell_2$};
			\draw(6.08,-2.19)node {$d_H(v_s)=\ell_s$};
		\end{tikzpicture}
		\caption{The multi-star $S_{(k;\ell_1,\ldots,\ell_s;\alpha_1,\ldots,\alpha_s)}(H)$}
		\label{Fig:multi-star}
	\end{center}
\end{figure}

The next lemma will be frequently used in subsequent proofs.
For $i\in [d]$, let $\mathbf{e}_i \in \{0,1\}^d$ be the vector whose $i$-th entry is $1$ and all other entries are $0$.

\begin{lemma}\label{change_b}
	Let $G$ be an $n$-vertex $d$-regular multigraph and $H$ be a spanning subgraph of $G$.
	Then the following hold:
	\begin{enumerate}[(1)]
		\item If $xy\in E(H)$ is an $H(i,j)$-edge, then $\mathbf{b}(H-xy)= \mathbf{b}(H) + \mathbf{e}_{i} + \mathbf{e}_{j}$.
		
		\item If $uv \in E(G\backslash H)$ is an $H(i,j)$-edge, then $\mathbf{b}(H+uv) = \mathbf{b}(H) - \mathbf{e}_{i+ 1} - \mathbf{e}_{j + 1}$.
		
		\item Let $H'$ be obtained from $H$ by deleting the edges of a copy of $S_{(k; \ell_1, \ldots, \ell_s;\alpha_1,\ldots,\alpha_s)}(H)$. Let $m=\sum_{i=1}^s \alpha_i$.
		Then $\mathbf{b}(H') = \mathbf{b}(H) +\sum_{i=1}^{m}\mathbf{e}_{k + 1 - i} + \sum_{i=1}^s \sum_{j=1}^{\alpha_i} \mathbf{e}_{\ell_i+1-j}$.
		
		\item Let $H''$ be obtained from $H$ by adding the edges of a copy of $\overline{S}_{(k; \ell_1, \ldots, \ell_s;\alpha_1,\ldots,\alpha_s)}(H)$.
		Let $m=\sum_{i=1}^s \alpha_i$.
		Then $\mathbf{b}(H'') = \mathbf{b}(H) - \sum_{i=1}^{m}\mathbf{e}_{k + i} - \sum_{i=1}^s \sum_{j=1}^{\alpha_i} \mathbf{e}_{\ell_i+j} $.
	\end{enumerate}
\end{lemma}

\begin{proof}
	To prove (1), we first consider the change of $x$'s degree after the deletion of the edge $xy$. We see $a_i=a_{d_H(x)}$ decreases by $1$, and $a_{i-1}=a_{d_H(x) - 1}$ increases by $1$. Hence $b_{i}$ increases by $1$ and all other $b_k$'s (for $k\neq i$) remain the same. Similarly, considering the change of $y$'s degree, we have $b_{j}$ increases by $1$ and the others remain the same. Hence we have prove $(1)$.
	
	The proof of (2) is similar. After adding $uv$ to $H$, both $d(u)$ and $d(v)$ increase by $1$. Hence $b_{d(u)+1}(H)$ and $b_{d(v)+1}(H)$ decrease by $1$ respectively, while other $b_k$'s remain the same.
	
	For the proof of (3),
	let $u$ be the center and $v_1,\ldots, v_s$ be distinct leaves of the deleted multi-star, where $d_H(u) = k$, $d_H(v_i)=\ell_i$ and there are $\alpha_i$ multi-edges between $u$ and $v_i$ for each $i\in [s]$.
	Then $H'$ can be generated from $H$ by recursively deleting the $m$ edges $(uv_1)^1, \ldots, (uv_1)^{\alpha_1}, \ldots, (uv_s)^1, \ldots, (uv_s)^{\alpha_s}$ in order.
	Let $H^{(i)}$ be obtained from $H$ by deleting the first $i$ edges. In this process of evolution (from $H=H^{(0)}$ to $H'=H^{(m)}$),
	the degree of $u$ is updated from $k$ to $k-m$, while the degree of $v_i$ is updated from $\ell_i$ to $\ell_i-\alpha_i$ for $i\in [s]$. Using (1), it is easy to obtain the desired equation as follows
	\begin{align*}
		\mathbf{b}(H') =
		\mathbf{b}(H^{(0)}) + \sum_{i=1}^m(\mathbf{b}(H^{(i)})-\mathbf{b}(H^{(i-1)}))=\mathbf{b}(H) + \sum_{i=1}^{m}\mathbf{e}_{k + 1 - i} + \sum_{i=1}^s \sum_{j=1}^{\alpha_i} \mathbf{e}_{\ell_i+1-j}.
	\end{align*}
	
	By utilizing (2), we can derive the proof of (4) in a similar manner.
\end{proof}

\section{General regular multigraphs: Theorem~\ref{thm1}}\label{Sec:thm1}
This section is devoted to the proof of Theorem~\ref{thm1}.
Let $G$ be an $n$-vertex $d$-regular multigraph throughout this section.
We aim to show that there exists a spanning subgraph $H$ of $G$ such that
$\norm{\mathbf{a}(H)}_{\infty}\leq 2d^2$ and in view of \eqref{equ:a<2b}, it suffices to show that  $\norm{\mathbf{b}(H)}_{\infty}\leq d^2$.
Along the way, we explain that this proof can be turned into a deterministic algorithm with time complexity $O_d(n^{d+1})$.

First, we introduce a concept that help describing the process of updating spanning subgraphs of $G$.
We define $\mathbf{C}(H)$ as the increasing ordering of the sequence $\{|b_i(H)|\}_{i=1}^d$.

\begin{definition}\label{Dfn:improv}
Let $H, H'$ be two spanning subgraphs of $G$.
We say $H'$ is an {\bf improvement} of $H$ in $G$, if they satisfy that either
\begin{itemize}
	\item [(A).]  $\sum_{j=1}^d |b_j(H')| < \sum_{j=1}^d |b_j(H)|$, or
	\item [(B).] $\sum_{j=1}^d |b_j(H')|=\sum_{j=1}^d |b_j(H)|$ but $\mathbf{C}(H) < \mathbf{C}(H')$ holds in the lexicographic order.
\end{itemize}
\end{definition}

The following lemma is key for Theorem~\ref{thm1}, whose proof will be postponed in Subsections~\ref{subsec:cons-multi-stars} and \ref{subsec:gene-impr}.

\begin{lemma}\label{lem:reduce}
Let $d\geq 2$. Let $G$ be an $n$-vertex $d$-regular multigraph and $H$ be its spanning subgraph. Let $M\geq d$ be a constant.
If there exists some $i\in [d]$ with $|b_i(H)|>M$, then
\begin{itemize}
	\item either there exists some $j\in [d]$ with $|b_j(H)|\in (M-d, M]$,
	\item or one can construct an improvement of $H$ in $G$, using linear time $O_d(n)$.
\end{itemize}
\end{lemma}

In the next lemma, we show that the number of successive improvements one can possibly perform is finitely bounded.

\begin{lemma}\label{num_op}
Let $G$ be an $n$-vertex $d$-regular multigraph. Then any sequence $(H_0, H_1, H_2, \ldots)$ of spanning subgraphs of $G$ such that $H_{i+1}$ is an improvement of $H_i$ for each $i\geq 0$ has length at most $O_d(n^d)$.
\end{lemma}
\begin{proof}
Let $H$ be any spanning subgraph of $G$.
It is evident that $\sum_{j=1}^d |b_j(H)| \leq dn$.
For $j\in [d]$, let $c^{(j)}(H)$ be the $j$-th smallest number in the sequence $\{|b_i(H)|\}_{i=1}^d$.
Then we have
\begin{equation}\label{equ:c^j}
	0\leq c^{(j)}(H) \leq n.
\end{equation}
Let $H_{i+1}$ be the improvement of $H_i$ in the sequence.
We say that $H_{i+1}$ has {\it type $0$} if $\sum_{j=1}^d |b_j(H_{i+1})| < \sum_{j=1}^d |b_j(H_i)|$,
and {\it type $k$} if $\sum_{j=1}^d |b_j(H_{i+1})|=\sum_{j=1}^d |b_j(H_i)|$, $c^{(j)}(H_{i})=c^{(j)}(H_{i+1})$ for every $j\leq k-1$ and $c^{(k)}(H_i)<c^{(k)}(H_{i+1})$.
Here, $k$ ranges between $1$ and $d-1$.

We point out that $b^{(j)}(H)\in \frac{1}{d+1}\mathbb{Z}$ for all $j$.
So $\sum_{j=1}^d |b_j(H)|\in \frac{1}{d+1}\mathbb{Z}$,
and this summation is strictly decreasing in all spanning subgraphs of type $0$.
This implies that in any considered sequence $(H_0, H_1, H_2, \ldots)$ of spanning subraphs of $G$,
there are at most $(d+1)dn+1$ spanning subgraphs of type $0$.
We can also derive from \eqref{equ:c^j} that between two subgraphs of type $0$, where no other subgraphs of type $0$ are in between,
there are at most $(d+1)n+1$ subgraphs of type $1$.
Similarly, between two subgraphs of type $i$, where no other subgraphs of type $j$ for some $j\leq i$ are in between, there are at most $(d+1)n+1$ subgraphs of type $i+1$.
Putting the above all together, we see that the length of any sequence $(H_0, H_1, H_2, \ldots)$ is at most
\begin{equation*}
	\big((d+1)dn + 1\big)\cdot \big((d+1)n+ 1\big)^{d-1} = O_d(n^d).
\end{equation*}
We have completed the proof of this lemma.
\end{proof}

With the two aforementioned lemmas in mind, we are now ready to provide the proof of Theorem~\ref{thm1}.

\begin{proof}[\bf Proof of Theorem~\ref{thm1} (assuming Lemma~\ref{lem:reduce}).]
Let $n>d\geq 2$ and $G$ be any $n$-vertex $d$-regular multigraph.
Our goal is to find a spanning subgraph $H$ of $G$ in time complexity $O_d(n^{d+1})$ with
$\norm{\mathbf{a}(H)}_{\infty}\leq 2d^2$.

By the discussion at the beginning of this section,
it suffices to find such $H$ with $\norm{\mathbf{b}(H)}_{\infty}\leq d^2$.
We make the following claim.

\medskip
{\noindent \bf Claim.} Given any spanning subgraph $H$ of $G$, either $\norm{\mathbf{b}(H)}_{\infty}\leq d^2$, or one can find an improvement of $H$ in $G$ using linear time $O_d(n)$.

\medskip

\noindent To prove this claim, we repeatedly use Lemma~\ref{lem:reduce}.
Let $M_t=td$ for integers $t\geq 0$.
Suppose that $\norm{\mathbf{b}(H)}_{\infty}> d^2= M_{d}$ (so there exists some $j_d\in [d]$ with $|b_{j_d}(H)|>M_d$).
By Lemma~\ref{lem:reduce}, either there exists some $j_{d-1}\in [d]$ with $|b_{j_{d-1}}(H)|\in (M_{d-1}, M_d]$,
or one can find an improvement of $H$ in $G$ using linear time $O_d(n)$.
In the latter case, this claim becomes valid.
So we may assume the former case occurs.
For every $d-1\geq i\geq 1$,
applying Lemma~\ref{lem:reduce} for $M_i$ in a similar matter, we can conclude that there exists some $j_{i-1}\in [d]$ with $|b_{j_{i-1}}(H)|\in (M_{i-1}, M_{i}]$ (as otherwise, one can find an improvement of $H$ in $G$ using linear time $O_d(n)$).
Putting these all together, we obtain $d+1$ indices $j_d, j_{d-1},\ldots,j_0\in [d]$ which are obviously all distinct.
This is a contradiction and thus proves the claim.

\medskip

Now using the above claim, we can achieve our goal via the following algorithm:
Initially choose any spanning subgraph $H_0$ of $G$.
If $\norm{\mathbf{b}(H_i)}_{\infty}\leq d^2$ for some $i\geq 0$, then we terminate;
otherwise, by the above claim one can find an improvement $H_{i+1}$ of $H_i$ in $G$ using linear time $O_d(n)$, and then we iterate this process by considering $H_{i+1}$.
In view of Lemma~\ref{num_op}, this algorithm must terminate in $O_d(n^d)$ iterations and thus terminate in time complexity $O_d(n^{d+1})$.
When it terminates at some spanning subgraph say $H^*$, it is clear that we meet the desired requirement $\norm{\mathbf{b}(H^*)}_{\infty}\leq d^2$,
finishing the proof of Theorem~\ref{thm1}.
\end{proof}

\medskip

In the remainder of this section, we complete the proof of Lemma~\ref{lem:reduce} through two steps. 
\begin{itemize}
\item[1.] Firstly, we demonstrate that certain specific multi-stars exist when the running spanning subgraph $H$ is far from being irregular (see Lemma~\ref{d+-}). 
\item[2.] Secondly, we establish that any multi-star obtained in the previous step can be utilized in a local adjustment to construct an improvement of the current running spanning subgraph (see Lemmas~\ref{lem:H-edge} and \ref{lem:multistar}). 
\end{itemize}
These two steps will be covered in Subsections~\ref{subsec:cons-multi-stars} and \ref{subsec:gene-impr}, respectively. 
The proof of Lemma~\ref{lem:reduce} will be presented at the end of this section.

\subsection{Constructing multi-stars}\label{subsec:cons-multi-stars}
In this subsection, we prove the existence of suitable multi-stars (with additional restrictions) whenever the spanning subgraph $H$ deviates from being irregular.

The coming lemma is mainly derived using a double-counting argument.

\begin{lemma}\label{ext}
Let $G$ be a $d$-regular multigraph and $H$ be its spanning subgraph.
Suppose $A, B$ are two disjoint subsets of $\{0,1,\ldots,d\}$.
Then the following hold:
\begin{itemize}
	\item [(a)] Suppose for any $i\in A$ and $j\notin B$, there is no $H(i,j)$-edge in $H$.
	Let $n_i, m_j\geq 1$ be integers for $i\in A$ and $j\in B$ such that $n_i\leq i$ and $\sum_{i\in A} n_i |V_i^H| > \sum_{j\in B} m_j |V_j^H|$.
	Then there exists a multi-star $S_{(k; \ell_1, \ldots, \ell_s;\alpha_1,\ldots,\alpha_s)}(H)$ with $m_k+1$ edges satisfying that $k\in B$, $\ell_q\in A$ and $1\leq \alpha_q\leq n_{\ell_q}$ for each $q\in [s]$.
	
	\item [(b)] Suppose for any $i\in A$ and $j\notin B$, there is no $H(i-1,j-1)$-edge in $G\backslash H$.
	Let $n_i, m_j\geq 1$ be integers for $i\in A$ and $j\in B$ such that $n_i \leq d-i+1$ and
	$\sum_{i\in A} n_i |V_{i-1}^H| > \sum_{j\in B} m_j |V_{j-1}^H|$. Then there exists a multi-star $\overline{S}_{(k-1; \ell_1-1, \ldots, \ell_s-1;\alpha_1,\ldots,\alpha_s)}(H)$ with $m_k+1$ edges satisfying that $k\in B$, $\ell_q\in A$ and $1\leq \alpha_q\leq n_{\ell_q}$ for each $q\in [s]$.
\end{itemize}
\end{lemma}

\begin{proof}
Consider (a) first. Let $i\in A$ and $x\in V_i^H$. Note that any edge of $H$ incident to $x$ must be incident to some vertex in $\cup_{j\in B} V_j^H$.
Since $1\leq n_i\leq i$, we can select $n_i$ edges of $H$ incident to $x$ and orient them to be directed edges with head $x$. Let $E_x$ denote the set of these directed edges. Let $D$ be the directed multigraph induced by the edge set $\cup_{i\in A}\cup_{x\in V_i^H} E_x$.
Then all directed edges of $D$ are oriented from $\cup_{i\in A} V_i^H$ to $\cup_{j\in B} V_j^H$, where each $x\in V_i^H$ for $i\in A$ has out-degree exactly $n_i$.
It is straightforward to see that if there exists some vertex $y\in V_k^H$ for some $k\in B$ with in-degree at least $m_k+1$ in $D$, then we can find a desired multi-star $S_{(k; \ell_1, \ldots, \ell_s;\alpha_1,\ldots,\alpha_s)}(H)$, where $\ell_q\in A$ and $1\leq \alpha_q\leq n_{\ell_q}$ for each $q\in [s]$.
So we may assume that any vertex $y\in V_j^H$ for $j\in B$ has in-degree at most $m_j$ in $D$. Then this leads to that $\sum_{i\in A} n_i |V_i^H| =e(D)\leq \sum_{j\in B} m_j |V_j^H|,$ a contradiction to the given condition.

The item (b) can be proved similarly, by considering the edges of $G\backslash H$ instead of $H$. Let $i\in A$ and $x\in V_{i-1}^H$. So there are $d-i+1$ edges of  $G\backslash H$ incident to $x$. Since $n_i\leq d-i+1$, among them we can select $n_i$ edges and orient these edges with head $x$. Again, let $E_x$ denote the set of these directed edges. Let $D'$ be the directed multigraph induced by the edge set $\cup_{i\in A}\cup_{x\in V_{i-1}^H} E_x$.
Using the same arguments for double-counting $e(D')$ as in the item $(a)$, we can finish the proof for the item $(b)$.
\end{proof}

We point out that the extra requirements $\alpha_q\leq n_{\ell_q}$ will be crucial in controlling the variations of irregularity vectors, while the practical meanings of the constants $n_i, m_j$ are revealed in Lemma~\ref{d+-}.
Next, we derive a special case for $d=3$ from Lemma \ref{ext}.

\begin{lemma}\label{find_edg}
Let $G$ be a cubic multigraph and $H$ be its spanning subgraph.
If $3a_3(H) > a_1(H) + 2a_2(H)$, then there is an $H(3,3)$-edge in $H$;
if $3a_0(H) > 2a_1(H) + a_2(H)$, then there is an $H(0,0)$-edge in $G\backslash H$.
\end{lemma}

\begin{proof}
Suppose that $3a_3(H) > a_1(H) + 2a_2(H)$ and $H$ has no $H(3,3)$-edges.
Let $A = \{3\}$ and $B =\{0,1,2\}$ with $n_3=3$ and $m_j=j$ for each $j\in B$.
Then $3|V_3^H| = 3a_3(H) + 3n/4 > (a_1(H) + n/4) + (2a_2(H) + 2n/4) = \sum_{j\in B}m_j|V_j^H|$. By Lemma \ref{ext}, there must be an $H(0,3)$-edge in $H$, a multi-star consisting of two edges $uv_1, uv_2\in E(H)$ with $d_H(u)=1$, or a multi-star consisting of three edges $uv_1, uv_2, uv_3\in E(H)$ with $d_H(u)=2$.
But clearly these three types of subgraphs do not exist in $H$, a contradiction.
This proves the first conclusion.
To show the second conclusion, we only need to consider $G\backslash H$ and use Proposition \ref{sym}.
\end{proof}

The goal of this subsection is accomplished by the following lemma.
It asserts that under some mild assumptions on the irregularity vector, one can always find appropriate multi-stars for the use of local adjustments.

\begin{lemma}\label{d+-}
Let $G$ be a $d$-regular multigraph and $H$ be its spanning subgraph.
For any fixed constant $\alpha>0$, let $A^+ = \{i\in [d]: b_i(H) > \alpha\}$ and $A^- = \{i\in[d]: b_i(H) < -\alpha\}$.
\begin{itemize}
	\item For $i\in A^+$, let $n_i, m_i>0$ be the smallest integers such that
	$i+n_i\notin A^+$ and $i-m_i\notin A^+$;
	\item For $i\in A^-$, let $n_i, m_i>0$ be the smallest integers such that
	$i-n_i \notin A^-$ and $i + m_i \notin A^-$.
\end{itemize}
If both $A^+$ and $A^-$ are not empty, then there exists at least one of the following subgraphs:
\begin{itemize}
	\item [(\rom{1}).] an $(i,j)$-edge in $H$ with $i\in A^-$ and $j\notin A^+$;
	\item [(\rom{2}).] an $(i-1,j-1)$-edge in $G\backslash H$ with $i \in A^+$ and $j\notin A^-$;
	\item [(\rom{3}).] a multi-star $S_{(k; \ell_1, \ldots, \ell_s;\alpha_1,\ldots,\alpha_s)}(H)$ with $m_k+1$ edges with $k\in A^+$, $\ell_q\in A^-$ and $1\leq \alpha_q\leq n_{\ell_q}$ for each $q\in [s]$;
	\item [(\rom{4}).] a multi-star $\overline{S}_{(k-1; \ell_1-1, \ldots, \ell_s-1;\alpha_1,\ldots,\alpha_s)}(H)$ with $m_k+1$ edges with $k\in A^-$, $\ell_q\in A^+$ and $1\leq \alpha_q\leq n_{\ell_q}$ for each $q\in [s]$.
\end{itemize}
\end{lemma}

\begin{proof}
Consider any interval $\{j,j+1,\ldots,k\}$ in $A^+$ with $j-1,k+1\notin A^+$.
Then by definition for any $j\leq i\leq k$, $n_i=k-i+1$ and $m_i=i-j+1$. Thus we have
\begin{align*}
	\sum_{i=j}^k (n_i-m_i)=\sum_{i=j}^k ((k-i+1) - (i-j+1))=\sum_{i=j}^k(k-2i+j)=0
\end{align*}
This can be easily generalized to the following	
\begin{equation}\label{n_i-m_i}
	\sum_{i\in A^+} (n_i-m_i) =0 \mbox{~~ and ~ similarly ~~} \sum_{i\in A^-} (n_i-m_i) =0
\end{equation}
Suppose for a contradiction that all items (\rom{1})-(\rom{4}) do not occur.
Then by Lemma~\ref{ext}, we have
\begin{equation}\label{ineq1}
	\sum_{i\in A^-} n_i|V_i^H|\leq \sum_{j\in A^+}m_j|V_j^H|
	\mbox{ ~~ and ~~ }
	\sum_{j\in A^+} n_j|V_{j-1}^H|\leq \sum_{i\in A^-} m_i |V_{i-1}^H|.
\end{equation}
Adding these two inequalities up, we can derive that
\begin{equation}\label{M}
	M:=\sum_{i \in A^-}(n_i|V_i^H| - m_i |V_{i-1}^H|) + \sum_{i\in A^+} (n_i|V_{i-1}^H| - m_i|V_i^H|)\leq 0
\end{equation}
In what follows, we will show that both above terms in $M$ are positive.
Therefore, $M$ is positive, which is a contradiction to \eqref{M} and thus completes the proof.

Let us first show that the second term is positive.
To show this, it suffices to show that for any interval $\{j,j+1,\ldots,k\}$ in $A^+$ with $j-1,k+1\notin A^+$, we have
$$T:=\sum_{i=j}^k \left(n_i|V_{i-1}^H| - m_i|V_i^H|\right)>0.$$
Indeed, by performing a detail calculation, we can verify this as follows.
Write $b_i:=b_i(H)$. Then $b_{i+1}-b_i=a_i(H)=|V_i^H|-\frac{n}{d+1}$.
Therefore,
\begin{align*}
	T &= \sum_{i=j}^k \left(n_i\big(b_i-b_{i-1}+\frac{n}{d+1}\big)-m_i\big(b_{i+1}-b_i+\frac{n}{d+1}\big)\right) \\
	&= \sum_{i=j}^k \bigg(n_i(b_i-b_{i-1})-m_i(b_{i+1}-b_i)\bigg)+\frac{n}{d+1}\cdot \sum_{i=j}^k (n_i-m_i)\\
	&= \sum_{i=j}^k \bigg( (k-j+2)b_i-(i-j+1)b_{i+1}-(k-i+1)b_{i-1}\bigg)\\
	&= 2(b_j+b_{j+1}+\ldots+b_{k}) -(k-j+1) (b_{k+1}+b_{j-1})> 0,
\end{align*}
where the third equation holds because we have \eqref{n_i-m_i}, $n_i=k-i+1$ and $m_i=i-j+1$ for any $j\leq i\leq k$, and the last inequality follows by the fact that $j-1, k+1\notin A^+$.

The proof of $\sum_{i \in A^-}(n_i|V_i^H| - m_i |V_{i-1}^H|)>0$ can be derived analogously. The only difference is to observe that for an interval $\{j,j+1,\ldots,k\}$ in $A^-$ with $j-1,k+1\notin A^-$, we have $n_i=i-j+1$ and $m_i=k-i+1$ for any $j\leq i\leq k$.
\end{proof}

To conclude this section, we would like to remark that there exists an algorithm (i.e., the greedy algorithm) with linear time complexity $O_d(|V(G)|)$ to compute $A^+, A^-$ as well as to find one of subgraphs listed in the four items of Lemma~\ref{d+-} (if it exists).

\subsection{Generating improvements from multi-stars}\label{subsec:gene-impr}
In the subsequent two lemmas, we show how the four subgraphs indicated by  Lemma~\ref{d+-} (under some mild conditions) help generating improvements of spanning subgraphs.

\begin{lemma}\label{lem:H-edge}
Let $G$ be a $d$-regular multigraph and $H$ be its spanning subgraph.
Then the following statements hold:
\begin{enumerate}[(1)]
	\item If there exists an $H(i,j)$-edge $e$ in $H$ with $b_i(H)<-1$ and $b_i(H)+b_j(H)<-1$,
	then $H-e$ is an improvement of $H$ in $G$;
	
	\item If there exists an $H(i-1,j-1)$-edge $e'$ in $G\backslash H$ with $b_i(H)>1$ and $b_i(H)+b_j(H)>1$,
	then $H+e'$ is an improvement of $H$ in $G$.
\end{enumerate}
\end{lemma}

\begin{proof}
Consider (1). Let $e\in E(H)$ be an $H(i,j)$-edge with $b_i(H)<-1$ and $b_i(H)+b_j(H)<-1$. Let $H'=H-e$.
By Lemma~\ref{change_b}, we have $\mathbf{b}(H') = \mathbf{b}(H) + \mathbf{e}_i + \mathbf{e}_j$.
Let us first assume that $b_j(H) < 0$.
If $i\neq j$, then it is clear that $\sum_{k=1}^d|b_k(H')| = \sum_{k=1}^d|b_k(H)| - 1 + (|b_j(H)+1| - |b_j(H)|) < \sum_{k=1}^d|b_k(H)|$;
otherwise $i=j$, then $b_i(H)=b_j(H) < -1$
and we also have $\sum_{k=1}^d|b_k(H')| = \sum_{k=1}^d|b_k(H)| + (|b_i(H) + 2| - |b_i(H)|) < \sum_{k=1}^d|b_k(H)|$.
This shows that $H'$ is an improvement of $H$ (with the property (A)).
It remains to consider when $b_j(H)\geq 0$.
In this case, we have $b_i(H')=b_i(H) + 1 <-b_j(H)\leq 0$ and $b_j(H') = b_j(H) + 1$.
Hence $\sum_{k=1}^d|b_k(H')| = \sum_{k=1}^d|b_k(H)|$, but both $|b_i(H')|$ and $|b_j(H')|$ are strictly bigger than $|b_j(H)|$.
Since all entries of $\mathbf{b}(H')$, except the $i$'th and $j$'th entries, remain the same as $\mathbf{b}(H)$,
we can easily derive that $\mathbf{C}(H) < \mathbf{C}(H')$ in the lexicographic order.
So $H'$ is an improvement of $H$ (with the property (B)), finishing the proof for (1).

To see (2), let $F=G\backslash H$. We observe that any $H(i-1,j-1)$-edge $e'$ in $F$ is also an $F(d-i+1,d-j+1)$-edge.
By Proposition~\ref{sym}, we also have $b_{d-i+1}(F)=-b_i(H)<-1$ and $b_{d-i+1}(F)+b_{d-j+1}(F)=-b_i(H)-b_j(H)<-1$.
Applying (1) with respect to $e'$ and $F$,
we conclude that $F-e'$ is an improvement of $F$.
By Proposition~\ref{sym} again, we observe that $\mathbf{C}(G\backslash G')=\mathbf{C}(G')$ holds for any subgraph $G'$ of $G$.
This fact implies that
$H+e'=G\backslash (F-e')$ is an improvement of $H=G\backslash F$. We have completed the proof of this lemma.
\end{proof}

\begin{lemma}\label{lem:multistar}
Let $d\geq 2$. Let $G$ be a $d$-regular multigraph and $H$ be its spanning subgraph. For a fixed constant $\alpha\geq d$, let $A^+, A^-$ and $n_i,m_i$ for $i\in A^+\cup A^-$ be defined the same as in Lemma~\ref{d+-}.
Then the following statements hold:
\begin{enumerate}[(1)]
	\item Assume that there exists a multi-star $S:=S_{(k; \ell_1, \ldots, \ell_s;\alpha_1,\ldots,\alpha_s)}(H)$ with $m_k+1$ edges with $k\in A^+$, $\ell_q\in A^-$ and $1\leq \alpha_q\leq n_{\ell_q}$ for each $q\in [s]$.
	If $b_{k-m_k}(H)\leq \alpha-d$,
	then $H-S$ is an improvement of $H$ in $G$;
	
	\item Assume that there exists a multi-star $\overline{S}:=\overline{S}_{(k-1; \ell_1-1, \ldots, \ell_s-1;\alpha_1,\ldots,\alpha_s)}(H)$ with $m_k+1$ edges with $k\in A^-$, $\ell_q\in A^+$ and $1\leq \alpha_q\leq n_{\ell_q}$ for each $q\in [s]$.
	If $b_{k+m_k}(H)\geq -(\alpha -d)$,
	then $H+\overline{S}$ is an improvement of $H$ in $G$.
\end{enumerate}
\end{lemma}

\begin{proof}
Suppose there exists a multi-star $S$ with $m_k+1$ edges and satisfying other conditions of (1).
Let $H'=H-S$.
By Lemma~\ref{change_b}, we have
$\mathbf{b}(H') = \mathbf{b}(H) +\mathbf{c}_1+\mathbf{c}_2+\mathbf{c}_3$,
where
$$\mathbf{c}_1=\sum_{i=1}^{m_k}\mathbf{e}_{k + 1 - i}, ~~ \mathbf{c}_2= \sum_{q=1}^s \sum_{j=1}^{\alpha_q} \mathbf{e}_{\ell_i+1-j}, ~~ \mbox{ and } ~~\mathbf{c}_3=\mathbf{e}_{k-m_k}.$$
By definition, we see that since $k\in A^+$,
we have $k, k-1,\ldots,k+1-m_k\in A^+$ but $k-m_k\notin A^+$;
since $\ell_q\in A^-$ and $1\leq \alpha_q\leq n_{\ell_q}$ for each $q\in [s]$,
we have $\ell_q, \ell_q-1,\ldots,\ell_q+1-\alpha_q \in A^-$.
These facts imply that after adding $\mathbf{c}_1$, the value of $\sum_{i=1}^d |b_i(\cdot)|$ increases by $m_k$, while after adding $\mathbf{c}_2$, the value of $\sum_{i=1}^d |b_i(\cdot)|$ decreases by $\sum_{q=1}^s \alpha_q=m_k+1$. 
Let $\mathbf{b}' = \mathbf{b}(H) + \mathbf{c}_1 + \mathbf{c}_2$, then 
$\sum_{i=1}^d|b'_i|=\sum_{i=1}^d|b_i(H)|+m_k-(m_k+1)=\sum_{i=1}^d|b_i(H)|-1$.
Note that $\mathbf{b}(H') = \mathbf{b}' + \mathbf{c}_3=\mathbf{b}'+\mathbf{e}_{k-m_k}$.

First we consider $k-m_k\in A^-$. 
Since 
$s \leq m_k+1 \leq |A^+|+1 \leq d$, we have
$b'_{k-m_k}\leq b_{k-m_k}(H)+s<-\alpha+d\leq 0$. 
Then
$$\sum_{i=1}^d|b_i(H')|=\sum_{i=1}^d|b'_i|+(|b'_{k-m_k}+1|-|b'_{k-m_k}|)<(\sum_{i=1}^d|b_i(H)|-1)+1=\sum_{i=1}^d|b_i(H)|,$$
implying that $H'$ is an improvement of $H$ (with property (A)).

Now we may assume that $k-m_k \notin A^-$, in this case $b'_{k-m_k} = b_{k-m_k}(H)$. If $b_{k-m_k}(H)<0$, then we can derive from the previous analysis that 
$$\sum_{i=1}^d |b_i(H')| <\sum_{i=1}^d |b_i(H)| ~~~~ \mbox{i.e., $H'$ is an improvement of $H$ (with property (A))}.$$
Otherwise $0\leq b_{k-m_k}(H)\leq \alpha-d$. 
In this case, we have $\sum_{i=1}^d |b_i(H')|=\sum_{i=1}^d |b_i(H)|$. 
Next we show that for all $j\in [d]$ for which $b_j(H')$ has been updated differently, it satisfies that
$$|b_j(H')|> |b_{k-m_k}(H)|.$$
This is clear for $j=k-m_k$ and for those $j\in A^+$; 
for those $j\in A^-$, since $s\leq m_k+1\leq d$,
we have $|b_j(H')|\geq |b_j(H)|-s>\alpha -d\geq |b_{k-m_k}(H)|$ as desired.
Hence, $\mathbf{C}(H) < \mathbf{C}(H')$ holds in the lexicographic order, i.e., $H'$ is an improvement of $H$ (with property (B)).
This proves the item (1).
The proof for the item (2) can be derived similarly so we omit here.
\end{proof}

Combining the above two lemmas with Lemma~\ref{d+-}, 
we are able to prove the key lemma for Theorem~\ref{thm1} as follows.

\begin{proof}[\bf Proof of Lemma~\ref{lem:reduce}.]
We assume that there is no $j\in [d]$ with $|b_j(H)|\in (M-d, M]$, and aim to show the existence of an improvement of $H$ in $G$.
The running time $O_d(n)$ will be discussed at the end of the proof.

By symmetry between $H$ and $G\backslash H$, we may assume that there exists some $i\in [d]$ with $b_i(H)> M$.
Set $\alpha:=M$.
Let $A^+=\{i\in[d]: b_i(H)>\alpha\}$ and
$A^-=\{i\in [d]: b_i(H)<-\alpha\}$.
So $A^+$ is nonempty.
Let $A^+ = \bigcup_{s=1}^\ell\{j_s,j_s+1,\ldots, k_s\}$, where each $k_s\leq j_{s+1}-2$. Then
\begin{equation*}
	\sum_{i \in A^+} |V_{i-1}(H)| = \sum_{i\in A^+} \big(b_{i}(H) - b_{i-1}(H)\big) + |A^+|\cdot n/(d + 1) >\sum_{s=1}^\ell (b_{k_s}(H)-b_{j_s-1}(H)) > 0.
\end{equation*}
This says, there exists a vertex of degree $i-1$ in $H$ for some $i\in A^+$.
Then there must be some $H(i-1,j-1)$-edge $e$ in $G\backslash H$ for some $i\in A^+$ and $j\in [d]$.
By Lemma~\ref{lem:H-edge}, since $b_i(H)>M>1$, if $b_i(H)+b_j(H)>1$, then $H+e$ is an improvement of $H$ in $G$.
So we may assume that $b_j(H)\leq -b_i(H)+1<-(M-1)$. 

If $-M\leq b_j(H) < -(M-1)$, then $|b_j(H)|\in (M-d,M]$, a contradiction of our assumption.
So we may assume that $b_j(H) < -M$, which implies that $A^-$ is also nonempty. By Lemma~\ref{d+-} (with $\alpha=M$),
there exists one of the four subgraphs listed from (\rom{1}) to (\rom{4}), denoted by $F$.
Suppose this subgraph $F$ is from (\rom{1}), i.e., an $H(i,j)$-edge $e^*$ in $H$ with $i\in A^-$ and $j\notin A^+$.
By our assumption that there is no $j\in [d]$ with $|b_j(H)|\in (M-d, M]$,
the fact that $j\notin A^+$ implies that $b_j(H)\leq M-d$.
Hence, we have $b_i(H)<-M<-1$ and $b_i(H)+b_j(H)<-M+(M-d)=-d\leq -1$.
By Lemma~\ref{lem:H-edge} (1), we see $H-e^*$ is an improvement of $H$ in $G$.
Similarly, if $F$ is from (\rom{2}), then one can use Lemma~\ref{lem:H-edge} (2) to find an improvement of $H$ in $G$.

Now suppose the subgraph $F$ is from (\rom{3}), i.e., a multi-star $S_{(k; \ell_1, \ldots, \ell_s;\alpha_1,\ldots,\alpha_s)}(H)$ with $m_k+1$ edges with $k\in A^+$, $\ell_q\in A^-$ and $1\leq \alpha_q\leq n_{\ell_q}$ for each $q\in [s]$.
Note that by definition, $k-m_k\notin A^+$,
implying that $b_{k-m_k}(H)\leq M-d=\alpha-d$, where $\alpha=M\geq d$. By Lemma~\ref{lem:multistar} (1),
we see $H-F$ is an improvement of $H$ in $G$.
It remains to consider the case when $F$ is from (\rom{4}) of Lemma~\ref{d+-}.
By similar arguments, we can use Lemma~\ref{lem:multistar} (2) to conclude that $H+F$ is an improvement of $H$ in $G$.
This proves the existence of an improvement of $H$.

By the remark after Lemma~\ref{d+-},
it takes linear time $O_d(n)$ to compute $A^+, A^-$ as well as to find one of subgraphs listed in the four items of Lemma~\ref{d+-} (if it exists).
Since all improvements of $H$ found in this proof are defined using $H$ and these subgraphs solely, it is evident that one can find such an improvement of $H$ using linear time $O_d(n)$.
\end{proof}

\section{Cubic multigraphs: Theorem~\ref{thm2}}\label{Sec:thm2}
In this section, we prove Theorem~\ref{thm2}, by showing that any cubic multigraph $G$ on $n$ vertices contains a spanning subgraph $H$ satisfying that $\norm{\mathbf{a}(H^*)}_{\infty}\leq 2$.
In fact we shall prove a slightly stronger statement (see Theorem~\ref{thm2:stronger}).

Throughout this section,
let $\mathcal{G}$ be the family of all cubic multigraphs.
Let $K_2^k$ denote the multigraph consisting of $k$ parallel edges between two fixed vertices.
In particular, $K_2^3\in \mathcal{G}$.

\subsection{Characterization and proof outline of cubic multigraphs}\label{subsec:char-cubic}
In this subsection, our objective is to establish a comprehensive structural description of all cubic multigraphs (refer to Lemma~\ref{cons}). 
Subsequently, we provide a concise proof outline of Theorem~\ref{thm2} based on this characterization.

The structural lemma states that any cubic multigraph can be generated from $mK_2^3$ (i.e., the disjoint union of $m$ copies of $K_2^3$'s) for some $m\geq 1$,
through a series of elementary operations which we define as follows
(also illustrated by Figure~\ref{Fig:operations}).

\begin{definition}\label{Def:cubic}
Let $G,G'\in \mathcal{G}$ satisfy that $|V(G')|=|V(G)|+2$.
We say that $G'$ is {\it generated} from $G$, denoted by $G \hookrightarrow G'$, if one of the following holds:
\begin{itemize}
	\item[(I).] There exists an edge $xy\in E(G)$ such that $G'$ is obtained from $G$ by removing the edge $xy$, introducing two new vertices $u$ and $v$, and adding the edges $xu$ and $vy$, along with two parallel edges between $u$ and $v$.
	\item[(II).] There exist two edges $xy, zw\in E(G)$ such that $G'$ is obtained from $G$ by removing the edges $xy$ and $zw$, introducing two new vertices $u$ and $v$, and adding the edges $xu, uy, zv, vw$ and $uv$.
\end{itemize}
\end{definition}

\begin{figure}[H]
\begin{center}
	\begin{tikzpicture}[thin][node distance=1cm,on grid]
		\node[circle,inner sep=0.5mm,draw=black,fill=black](x5)at(-3.0,0) [label=above:$x$] {};
		\node[circle,inner sep=0.5mm,draw=black,fill=black](x6)at(-1.4,0) [label=above:$y$] {};
		\node[circle,inner sep=0.5mm,draw=black,fill=black](x1)at(-0.2,0) [label=above:$x$] {};
		\node[circle,inner sep=0.5mm,draw=black,fill=black](x2)at(2.6,0) [label=above:$y$] {};
		\node[circle,inner sep=0.5mm,draw=black,fill=black](x3)at(1.2,-0.6) [label=below:$v$] {};
		\node[circle,inner sep=0.5mm,draw=black,fill=black](x4)at(1.2,0.6) [label=above:$u$] {};
		\draw[draw,thick,color=blue] (1.2,0.6) arc (150:210:1.2);
		\draw[draw,thick,color=blue] (1.2,0.6) arc (30:-30:1.2);
		\draw[thick,color = blue](x1)to(x4);
		\draw[thick,color = blue](x2) to (x3);
		\draw[draw,thick,color=orange] (-1.4,0) arc (60:120:1.6);
		\draw[->, thick] (-1.0,0) to (-0.6,0);
		\draw(-0.8,-1.4) node {Type (I)};
	\end{tikzpicture}
	\hspace{1.5cm}
	\begin{tikzpicture}[thin][node distance=1cm,on grid]
		\node[circle,inner sep=0.5mm,draw=black,fill=black](x1)at(-1.2,1/2) [label=above:$x$] {};
		\node[circle,inner sep=0.5mm,draw=black,fill=black](x2)at(1.2,1/2) [label=above:$y$] {};
		\node[circle,inner sep=0.5mm,draw=black,fill=black](x3)at(-1.2,-1/2) [label=below:$z$] {};
		\node[circle,inner sep=0.5mm,draw=black,fill=black](x4)at(1.2,-1/2) [label=below:$w$] {};
		\node[circle,inner sep=0.5mm,draw=black,fill=black](x5)at(0,1/2) [label=above:$u$] {};
		\node[circle,inner sep=0.5mm,draw=black,fill=black](x6)at(0,-1/2) [label=below:$v$] {};
		\node[circle,inner sep=0.5mm,draw=black,fill=black](x7)at(-4.2,1/2) [label=above:$x$] {};
		\node[circle,inner sep=0.5mm,draw=black,fill=black](x8)at(-2.6,1/2) [label=above:$y$] {};
		\node[circle,inner sep=0.5mm,draw=black,fill=black](x9)at(-4.2,-1/2) [label=below:$z$] {};
		\node[circle,inner sep=0.5mm,draw=black,fill=black](x10)at(-2.6,-1/2) [label=below:$w$] {};
		\draw[thick,color = blue](x1)to(x5);
		\draw[thick,color = blue](x2)to(x5);
		\draw[thick,color = blue](x3)to(x6);
		\draw[thick,color = blue](x4)to(x6);
		\draw[thick,color = blue](x5)to(x6);
		\draw[draw,thick,color=orange] (-2.6,1/2) arc (60:120:1.6);
		\draw[draw,thick,color=orange] (-2.6,-1/2) arc (-60:-120:1.6);
		\draw[->, thick] (-2.1,0) to (-1.7,0);
		\draw(-1.9,-1.4) node {Type (II)};
	\end{tikzpicture}
	\caption{The operation $G \hookrightarrow G'$}
	\label{Fig:operations}
\end{center}
\end{figure}

We shall point out that the two edges $xy, zw$ in Type (II) may share some common vertex or even be parallel.
The next lemma indicates that for any graph $G'\in \mathcal{G}$ which is not a disjoint union of copies of $K_2^3$'s, there exists some $G\in \mathcal{G}$ such that $G\hookrightarrow G'$.

\begin{lemma}\label{ctr}
Let $G'$ be any connected cubic multigraph with at least four vertices.
Then there exists some $G\in \mathcal{G}$ such that $G\hookrightarrow G'$.
Moreover, if the multiplicities of all edges are provided, 
such $G$ can be constructed with a time complexity of $O(1)$.
\end{lemma}

\begin{proof}
We first claim that one can find an edge $e=uv\in E(G')$ such that
\begin{itemize}
	\item[(1)] if $e$ has multiplicity one, then $|N_{G'-e}(u)| = |N_{G'-e}(v)| = 2$;
	\item[(2)] otherwise, $e=uv$ has multiplicity two and $N_{G'}(u)\backslash\{v\} \neq N_{G'}(v)\backslash \{u\}$.
\end{itemize}
If $G'$ is a simple cubic graph, then every edge $uv \in E(G')$ satisfies item (1), so the conclusion holds.
Hence we may assume that $G$ has some edges with multiplicity at least two.
Choose any such edge $e_1=uv$.
The multiplicity of $e_1$ is two (as otherwise, it must be three and then it contradicts the fact that $G'$ is connected and has at least four vertices).
Then either $e_1$ satisfies item (2), or there exists some vertex $w\in N_{G'}(u)\cap N_{G'}(v)$.
We assume the latter case holds.
Then there exists a unique edge $wx$ incident to $w$ with $x\notin \{u,v\}$ (as otherwise, the degree of $u$ or $v$ would be four in $G'[\{u,v,w\}]$).
Also $wx$ has multiplicity one.
The only possibility that $wx$ does not satisfy item (1) is when there are two parallel edges between $x$ and some new vertex $y\notin \{u,v,w,x\}$.
But then one of the parallel edges $xy$ satisfies item (2) as the other neighbour of $y$ cannot be $w$ (this is simply because we have seen $\{u,v,x\}\subseteq N_{G'}(w)$). 
This proves the claim.
We point out that the edge $e_1$ can be chosen arbitrarily, and all the mentioned edges and vertices are within a constant distance from $e_1$.
Therefore, if the multiplicities of all edges are provided,
then the proof of the above claim can be executed in $O(1)$ time.

Consider the case when $e=uv$ has multiplicity two.
Let $x\neq y$ be the two distinct vertices satisfying that $\{x\}=N_{G'}(u)\backslash\{v\}$ and $\{y\}=N_{G'}(v)\backslash\{u\}$.
Let $G = G'-\{u,v\} + xy$. Then clearly $G$ is a cubic multigraph, and $G'$ can be generated from $G$ in type (I) in $O(1)$ time.

Now suppose that $e=uv$ has multiplicity one.
Let $\{x,y\} = N_{G'-e}(u)\backslash \{v\}$ and $\{z,w\} = N_{G'-e}(v)\backslash\{u\}$.
Note that $\{x,y\}$ and $\{z,w\}$ may overlap.
Let $G$ be obtained from $G'$ by deleting the vertices $u,v$ and adding two new edges $xy$ and $zw$.
Since $x\neq y$ and $z\neq w$,
$G$ is also a cubic multigraph, and $G'$ can be generated from $G$ in type (II),
finishing the proof.
\end{proof}

Next, we prove the main result of this subsection.
Note that any graph in $\mathcal{G}$ has an even number of vertices.

\begin{lemma}\label{cons}
Let $G\in \mathcal{G}$ be any multigraph on $2n$ vertices.
Then there exists a sequence $G_0, G_1, \ldots, G_k$ of cubic multigraphs such that $G_0=m K_2^3$, $G_k=G$, and $G_i\hookrightarrow G_{i+1}$ for every $0\leq i\leq k-1$.
Here, we have $n=m+k$.
\end{lemma}
\begin{proof}
We prove this by induction on $n$.
This statement is trivial when $n=1$.
Now suppose this holds for any graph in $\mathcal{G}$ with less than $2n$ vertices.
Let $G\in \mathcal{G}$ have $2n$ vertices.
If $G$ itself is a disjoint union of copies of $K_2^3$'s, then the conclusion holds trivially.
Otherwise, $G$ has a connected component $F\in \mathcal{G}$ with at least four vertices.
By Lemma~\ref{ctr}, there exists $F^*\in \mathcal{G}$ such that $F^*\hookrightarrow F$.
Then we also have $G^*\hookrightarrow G$, where $G^*\in \mathcal{G}$ denotes the disjoint union of $G-V(F)$ and $F^*$.
Note that $|V(G^*)|=|V(G)|-2=2n-2$.
By applying induction to $G^*$, it becomes evident that, along with $G^*\hookrightarrow G$, the induction provides a desired sequence of cubic multigraphs for $G$. 
\end{proof}

\medskip

Now we provide an outline of the proof for Theorem~\ref{thm2}. 
Consider an arbitrary cubic multigraph $G$. 
The overall approach (see Figure~\ref{overall_proof_cubic}) involves recursively growing subgraphs based on the sequence $G_0 = mK_2^3, G_1, \ldots, G_k = G$ as provided by Lemma~\ref{cons}.
The crux of the proof lies in an inductive argument: given a ``suitable'' subgraph $H_i$ of $G_i$ and the operation $G_i \hookrightarrow G_{i+1}$, our aim is to find a ``suitable'' subgraph $H_{i+1}$ of $G_{i+1}$. 
This inductive argument, as illustrated in Figure~\ref{fig:induction}, involves several reductions among various states of spanning subgraphs. 
The precise definitions of these states, which are technical in nature (see Definition~\ref{def:states}), will be provided in the subsequent subsection.

\renewcommand{\arraystretch}{1.5}
\setlength{\arraycolsep}{2.5pt}
\begin{figure}[H]
\begin{minipage}{.47\textwidth}
	\centering
	\[
	\begin{matrix}
		G_0 =mK_2^3 & \hookrightarrow & G_1 &\hookrightarrow & \ldots & \hookrightarrow &G_k=G \\  
		\phantom{G=m}\scalebox{1.2}{\rotatebox{90}{\makebox(0,0){ \  \ensuremath{\subseteq }}}} & & \scalebox{1.2}{\rotatebox{90}{\makebox(0,0){ \  \ensuremath{\subseteq }}}}& & & & \scalebox{1.2}{\rotatebox{90}{\makebox(0,0){ \  \ensuremath{\subseteq }}}}  \phantom{\;=G} \\ 
		\phantom{G=m}H_0 & \mapsto & H_1 & \mapsto & \ldots & \mapsto & H_k \phantom{\;=G}
	\end{matrix}
	\]
	\caption{The overall proof of Theorem \ref{thm2}}
	\label{overall_proof_cubic}
\end{minipage}
\hfill
\begin{minipage}{.47\textwidth}
	\centering
	\begin{tikzcd}[column sep=2.7cm, row sep=1.23cm]
		\text{State-1} \arrow[r,"\text{\footnotesize Lemma \ref{state1}}"] & \text{Proper} \arrow[d,"\text{\footnotesize Lemma \ref{lem:state-0}}"]\\
		\text{State-2} \arrow[u,swap, "\text{\footnotesize Lemma \ref{state2}}"] & \text{State-0}\arrow[l,dashed,""{name=U, label=above:\text{\footnotesize Lemma \ref{op1}}, label=below:\text{\footnotesize (Inductive Step)}}]
	\end{tikzcd}
	\caption{Proof of $H_i \mapsto H_{i+1}$ as in Figure~\ref{overall_proof_cubic}}
	\label{fig:induction}
\end{minipage}
\end{figure}
\setlength{\arraycolsep}{3pt}
\renewcommand{\arraystretch}{1.0}

\subsection{Reduction among states of spanning subgraphs}
In the upcoming lemmas, we present a series of reductions involving spanning subgraphs with distinct properties, which we refer to as ``states''. 
These reductions are depicted in Figure~\ref{fig:induction}. 

Before proceeding, let us provide the formal definitions of these states.

\begin{definition}\label{def:states}
Let $H$ be a spanning subgraph of a multigraph $G\in \mathcal{G}$. Then we say
\begin{itemize}
	\item $H$ is in \textbf{state-$2$} if $a_0(H) = -\frac52$, $a_1(H)\in [\frac{1}{2},\frac52]$, $a_2(H)\in [-\frac12, \frac32]$ and $a_3(H)\in [-\frac32, \frac12]$,
	\item $H$ is in \textbf{state-$1$} if $a_0(H)= -\frac52$ and $a_1(H), a_2(H), a_3(H)\in [-\frac12, \frac32]$,
	\item $H$ is \textbf{proper} if $a_i(H)\in [-2,2]$ holds for each $0\leq i \leq 3$ (i.e., $\norm{\mathbf{a}(H)}_{\infty}\leq 2$), and
	\item $H$ is in \textbf{state-$0$} if $a_0(H), a_3(H)\in [-2,\frac12]$ and $a_1(H), a_2(H)\in [-\frac12,2]$.
\end{itemize}
\end{definition}

We collect some basic facts on cubic multigraphs in the following.

\begin{proposition}\label{prop:states}
Let $H$ be a spanning subgraph of an $n$-vertex multigraph $G\in \mathcal{G}$. Then
\begin{itemize}
	\item[(1)] For all $0\leq i\leq 3$, $a_i(H)\in \frac{1}{2}\mathbb{Z}$.
	In particular, if $n=4k$, then all $a_i(H)$ are integers;
	and if $n=4k+2$, then all $a_i(H) \in \mathbb{Z} + \frac{1}{2}$.
	\item[(2)] $H$ is in state-0 (or proper) if and only if $G\backslash H$ is in state-0 (or proper).
\end{itemize}
\end{proposition}

\begin{proof}
The first item follows easily by the definition of $a_i(H)$, and the last item holds because of that $a_i(H)=a_{3-i}(G\backslash H)$ given by Proposition \ref{sym}.
\end{proof}

To proceed, we need to introduce additional notation.
Let $H$ be a subgraph of a cubic multigraph $G$.
For integers $0\leq i,j\leq 3$,
let $\mathcal{P}_{ij}(H)$ denote the family of all $H(i,j)$-edges in $H$,
and let $\mathcal{Q}_{ij}(H)$ denote the family of all $H(i,j)$-edges in $G\backslash H$.
Also define $p_{ij}(H)=|\mathcal{P}_{ij}(H)|$ and $q_{ij}(H)=|\mathcal{Q}_{ij}(H)|$.
Note that we always have $p_{0j}(H)=0$ and $q_{3j}(H)=0$ for any $j$.

\begin{lemma}\label{lem:state-0}
If a cubic multigraph $G$ has a proper subgraph $H$, 
then it has a subgraph in state-$0$.
In addition, if $\mathcal{P}_{ij}(H)$ and $\mathcal{Q}_{ij}(H)$ for all $0\leq i,j\leq 3$ are provided,
then one can construct a subgraph of $G$ in state-$0$ using $O(1)$ time.
\end{lemma}

\begin{proof}
Fix $G\in \mathcal{G}$. This proof will be divided into two steps:
First, we show that there exists a proper subgraph $H$ of $G$ satisfying that $a_0(H), a_3(H)\in [-2, \frac12]$.
Second, building on such a proper subgraph, we find a subgraph $H$ of $G$ in state-$0$.

To show the first step,
we choose $H$ to be a proper subgraph of $G$ with the minimum $f(H) := a_0(H) + a_3(H)$.
By Proposition~\ref{sym}, we have $f(G\backslash H) = f(H)$, so by symmetry, we may assume that $-2\leq a_0(H)\leq a_3(H)$.
If $a_3(H)\leq \frac12$, then we achieve the first step.
So we may assume $a_3(H) > \frac12$, and further by Proposition~\ref{prop:states}, $a_3(H)\geq 1$.
Suppose that $a_2(H) \leq 0$.
Then $3a_3(H) \geq 3 > 2a_2(H) + a_1(H)$.
By Lemma \ref{find_edg}, we can find an $H(3,3)$-edge $e\in E(H)$, i.e., $e\in \mathcal{P}_{33}(H)$.
Now consider $H' = H-e$, and we see that $a_3(H')=a_3(H)-2\in [-1,0]$, $a_2(H')=a_2(H)+2\in [0,2]$ and all other $a_j(H')=a_j(H)\in [-2,2]$ remain the same.
This implies that $H'$ is a proper subgraph of $G$ with $f(H')=f(H) - 2$,
which is a contradiction to the minimum of $f(H)$.		
Therefore, we may assume that $a_2(H) > 0$.
Since $a_2(H) + a_3(H)\in \mathbb{Z}$ and $a_2(H) + a_3(H) \geq 1 + \frac12 = \frac32$,
we derive that $a_2(H) + a_3(H) \geq 2$, which further implies that both $a_0(H), a_1(H) \leq 0$ (as otherwise $0=a_0(H) + a_1(H) + a_2(H) + a_3(H) > 0 + (-2) + 2 = 0$, a contradiction). To complete the proof of this step, we consider the following two cases:
\begin{itemize}
	\item There exists an $H(2,3)$-edge $e'\in E(H)$. Then $H' = H - e'$ satisfies $\mathbf{a}(H') = \mathbf{a}(H) + (0,1,0,-1)$,
	which implies that $H'$ is proper with $f(H') = f(H) - 1$, a contradiction.
	
	\item There does not exist $H(2,3)$-edges in $H$.
	Using $|V^H_i|=a_i(H)+\frac{n}{4}$, it is easy to see that $2|V^H_2|>|V^H_1|\geq p_{12}(H)$, and $3|V^H_3|>|V_1^H|\geq p_{13}(H)$.
	Also note that $p_{02}(H)=p_{03}(H)=p_{23}(H)=0$ and $i|V^H_i|=\sum_{j\in [3]\backslash \{i\}} p_{ij}(H)+2p_{ii}(H)$.
	These clearly imply the existence of an $H(2,2)$-edge $e_1\in E(H)$ and an $H(3,3)$-edge $e_2\in E(H)$.
	Then $H' = H - e_1 - e_2$ satisfies that $\mathbf{a}(H') = \mathbf{a}(H) + (0,2,0,-2)$.
	So $H'$ is proper with $f(H') = f(H) - 2$, again a contradiction to the minimum of $f(H)$.
\end{itemize}
Hence the proper subgraph $H$ we have chosen must satisfy that $a_0(H), a_3(H)\in [-2, \frac12]$.
Now we explain that provided $\mathcal{P}_{ij}(H)$ and $\mathcal{Q}_{ij}(H)$ for all $0\leq i,j\leq 3$, such a proper subgraph can be constructed in $O(1)$ time. 
The above proof can be transformed into an algorithm that generates a sequence of proper subgraphs $H$ iteratively, where each subsequent subgraph has a strictly decreasing $f(H)$ value.\footnote{Upon generating proper subgraphs $H$, the algorithm also updates $\mathcal{P}_{ij}(H)$ and $\mathcal{Q}_{ij}(H)$ for all $0\leq i,j\leq 3$.}
Additionally, each subsequent proper subgraph in this sequence is obtained from the current $H$ by deleting or adding one or two edges from specific families $\mathcal{P}_{ij}(H)$ or $\mathcal{Q}_{ij}(H)$. It is worth noting that the choice of edges from these families can be made arbitrarily. 
So each iteration only takes $O(1)$ time.\footnote{This $O(1)$ time complexity also includes the time required for updating all $\mathcal{P}_{ij}(H)$ and $\mathcal{Q}_{ij}(H)$.}
Since $-4\leq f(H)\leq 4$ holds for any proper $H$ and the value $f(H)$ in each iteration decrease at least $1$,
we conclude that this algorithm must terminate in at most 8 iterations and thus using $O(1)$ time.

For the second step, we choose a proper subgraph $H$ satisfying that $a_0(H), a_3(H)\in [-2, \frac12]$, $a_2(H) \leq a_1(H)$,\footnote{Here, we leverage the symmetry between $H$ and $G\backslash H$ as stated in Propositions~\ref{sym} and \ref{prop:states}.} and subject to the above conditions, $a_2(H)$ is maximum.
If $a_2(H)\geq -\frac12$, then such $H$ is in state-$0$ and we are done.
Hence we may assume that $a_2(H) < -\frac12$ (and so $a_2(H)\in [-2,-1]$).
In this case, $a_1(H)=-a_2(H) - a_3(H) - a_0(H)\geq 1 - \frac12 - \frac12 = 0$.

Suppose that $a_3(H) > 0$.
Then we have $a_1(H)+2a_2(H)\leq 2-2=0<3a_3(H)$,
so by Lemma~\ref{find_edg}, there exists an $H(3,3)$-edge $e\in E(H)$.
Then $H' = H-e$ satisfies that $\mathbf{a}(H') = \mathbf{a}(H) + (0,0,2,-2)$.
It is easy to verify that $H'$ is a subgraph in state-$0$.
Hence, we may assume that $a_3(H)\in [-2,0]$.
Recall that $a_1(H)\in [0,2]$ and $a_2(H)\in [-2,-1]$,
which imply that $a_0(H)\in [-1, \frac12]$.
It remains to consider two cases as follows.
\begin{itemize}
	\item There exists an $H(0,1)$-edge $e$ in $G\backslash H$.
	Then $H' = H+e$ satisfies
	$\mathbf{a}(H') = \mathbf{a}(H) + (-1,0,1,0)$.
	So $H'$ satisfies the same conditions as the chosen $H$, but $a_2(H')=a_2(H)+1>a_2(H)$.
	This is a contradiction to the choice of $H$.
	
	\item There does not exist $H(0,1)$-edges in $G\backslash H$.
	Recall the definition of $q_{ij}(H)$.
	Similar to the above discussion, we see that $3|V^H_0|>|V_2^H|\geq q_{02}(H)$, and $2|V^H_1|>|V^H_2|\geq q_{12}(H)$, from which we can find an $H(0,0)$-edge $e_1\in E(G\backslash H)$ and an $H(1,1)$-edge $e_2\in E(G\backslash H)$.
	If $a_0(H)\geq 0$, then $H'=H + e_1 + e_2$ satisfies
	$\mathbf{a}(H') = \mathbf{a}(H) + (-2,0,2,0)$ and thus is a subgraph in state-$0$.
	Otherwise, we have either $a_0(H)=-1$ or $-\frac12$.
	In the former case, it is clear that $\mathbf{a}(H) = (-1,2,-1,0)$;
	in the latter case, using Proposition~\ref{prop:states} (1),
	we can derive that $\mathbf{a}(H)=(-\frac12,\frac32,-\frac32,\frac12)$.
	In each of the two cases, $H+e_2$ satisfies $\mathbf{a}(H+e_2) = \mathbf{a}(H)+(0,-2,2,0)$ and thus is a subgraph in state-$0$.
\end{itemize}	
Now we have found a subgraph in state-$0$ as desired.
Based on similar explanations as given after the first step (considering that $-2\leq a_2(H)\leq 2$ is finite and bounded for any proper $H$), we can conclude that the second step can also be executed in $O(1)$ time, 
provided $\mathcal{P}_{ij}(H)$ and $\mathcal{Q}_{ij}(H)$ for all $0\leq i,j\leq 3$.
This finishes the proof of Lemma~\ref{lem:state-0}.
\end{proof}

The following two lemmas will employ similar arguments to those presented in Lemma~\ref{lem:state-0}.
For $i\in \{2,3\}$ and a subgraph $H$ of a cubic multigraph $G$,
let $S_{i;1,1}(H)$ denote a path $xyz$ contained in $H$, where $d_H(x)=d_H(z)=1$ and $d_H(y)=i$.
Furthermore, we define $\mathcal{S}_i(H)$ to be the family consisting of all $S_{i;1,1}(H)$. 

\begin{lemma}\label{state1}
If a cubic multigraph $G$ has a subgraph $H$ in state-$1$, then it has a proper subgraph.
In addition, if $\mathcal{P}_{ij}(H)$, $\mathcal{Q}_{ij}(H)$, $\mathcal{S}_i(H)$ and $\mathcal{S}_i(G\backslash H)$ for all $0\leq i,j\leq 3$ are provided,
then one can construct a proper subgraph of $G$ using $O(1)$ time.
\end{lemma}

\begin{proof}
Let $H$ be a subgraph of $G$ in state-$1$. Then $a_0(H)=-\frac{5}{2}$.
By Proposition~\ref{prop:states}, we see $a_i(H)\in \mathbb{Z}+\frac12$ for all $i$.
So $a_1(H),a_2(H),a_3(H)\in \{-\frac12,\frac12, \frac32\}$.

To turn such $H$ into a proper subgraph, our first attempt is to increase the value of $a_0(\cdot)$ by finding some $H(1,j)$-edge in $H$ for some $j$ and then removing it.
Indeed, if there exists an $H(1,2)$-edge $e\in E(H)$, then $H' = H - e$ satisfies that $\mathbf{a}(H') = \mathbf{a}(H) + (1,0,-1,0)$ and thus is a proper subgraph.
So we may assume $p_{12}(H)=0$.

Next, suppose that there exists an $H(1,1)$-edge $f\in E(H)$.
Let $H' = H - f$. Then $\mathbf{a}(H') = \mathbf{a}(H) + (2,-2,0,0)$.
So $H'$ is not proper if and only if $a_1(H) = -\frac12$.
This case forces that $\mathbf{a}(H) = (-\frac52,-\frac12,\frac32,\frac32)$ and $\mathbf{a}(H') = (-\frac12,-\frac52,\frac32,\frac32)$.
Since $p_{12}(H')+p_{23}(H')+2p_{22}(H')=2|V_2^{H'}| > |V_1^{H'}|\geq p_{12}(H')$,
there exists either an $H'(2,2)$-edge $e_1$ or an $H'(2,3)$-edge $e_2$ in $H'$.
Let $H_1 = H' - e_1$ and $H_2 = H' - e_2$.
As $\mathbf{a}(H_1) = (-\frac12,-\frac12,-\frac12,\frac32)$ and $\mathbf{a}(H_2) = (-\frac12,-\frac32,\frac32,\frac12)$,
at least one of $H_1, H_2$ exists which is a proper subgraph of $G$.
So we may also assume $p_{11}(H)=0$.

Since $p_{11}(H)=p_{12}(H)=0$, there must be an $H(1,3)$-edge $e$ in $H$.
Let $H' = H - e$. Then $\mathbf{a}(H') = \mathbf{a}(H) + (1,-1,1,-1)$.
It is clear that $H'$ is not proper if and only $a_2(H) = \frac32$.
Assume $a_2(H) = \frac32$ occurs.
Then $a_1(H)+a_3(H)=1$, where $a_1(H), a_3(H)\in \{-\frac12,\frac12, \frac32\}$.
It follows that either (I) $a_1(H) = \frac32$ and $a_3(H) = -\frac12$, or (II) $a_1(H)\leq \frac12 \leq a_3(H)$.

Consider the case (I) $a_1(H) = \frac32$ and $a_3(H) = -\frac12$. 
In this case since $|V_1^H| > |V_3^H|$, by (a) of Lemma~\ref{ext} (by choosing $A=\{1\}, B=\{3\}$ and $n_1=m_3=1$), there exists some $S_{(3;1,1)}(H)$. 
Let $H_1$ be obtained from $H$ by deleting the two edges of $S_{(3;1,1)}(H)$.
Then we have $\mathbf{a}(H_1) = \mathbf{a}(H) + (2,-1,0,-1) = (-\frac12,\frac12,\frac32,-\frac32)$, 
which shows that $H_1$ is proper.

Now consider the case (II) $a_1(H)\leq \frac12 \leq a_3(H)$. 
Suppose that there exists an $H(2,2)$-edge $f$ in $H$.
Let $H_1 = H - f - e$, where $e$ is the $H(1,3)$-edge in $H$ mentioned above.
Then $\mathbf{a}(H_1) = \mathbf{a}(H) +(1,1,-1,-1)$, implying that $H_1$ is a proper subgraph of $G$.
Hence, we may assume that $p_{22}(H)=0$. 
Recall that we also have $p_{12}(H)=0$.
So there must exist some $H(2,3)$-edge $f'\in E(H)$. 
Let $H' = H - f'$. Then $\mathbf{a}(H') = \mathbf{a}(H) + (0,1,0,-1)$.
So $H'$ remains in state-$1$, with $a_1(H')=a_1(H)+1$ and $a_3(H')=a_3(H)-1$.
Repeating exactly the same proof as above on $H'$,
we can find either a proper subgraph of $G$, or find an $H'(2,3)$-edge $f''$ in $H'$
such that $H''=H'-f''$ is also in state-$1$ and satisfying $a_1(H'')=\frac32$ and $a_3(H'')=-\frac12$, i.e., the case (I).
As shown above, the latter case can lead to a proper subgraph of $G$.
So in any situation, $G$ has a proper subgraph. 

Finally, we point out that if $\mathcal{P}_{ij}(H)$, $\mathcal{Q}_{ij}(H)$, $\mathcal{S}_i(H)$ and $\mathcal{S}_i(G\backslash H)$ for all $0\leq i,j\leq 3$ are provided,
then similar to Lemma~\ref{lem:state-0}, the above arguments also can be implemented in $O(1)$ time because all local adjustments are obtained by adding or deleting finite members from $\mathcal{P}_{ij}(H)$, $\mathcal{Q}_{ij}(H)$, $\mathcal{S}_i(H)$ and $\mathcal{S}_i(G\backslash H)$.
This finishes the proof. 
\end{proof}

\begin{lemma}\label{state2}
If a cubic multigraph $G$ has a subgraph $H$ in state-$2$, then it has a proper subgraph.
In addition, if $\mathcal{P}_{ij}(H)$, $\mathcal{Q}_{ij}(H)$, $\mathcal{S}_i(H)$ and $\mathcal{S}_i(G\backslash H)$ for all $0\leq i,j\leq 3$ are provided,
then one can construct a proper subgraph of $G$ using $O(1)$ time.
\end{lemma}

\begin{proof}
Let $H$ be a subgraph in state-$2$.
That is, $a_0(H) = -\frac52$, $a_1(H)\in [\frac{1}{2},\frac52]$, $a_2(H)\in [-\frac12, \frac32]$ and $a_3(H)\in [-\frac32, \frac12]$.
By Proposition~\ref{prop:states}, we see $a_i(H)\in \mathbb{Z}+\frac12$ for all $i$.
If $a_1(H)\leq \frac32$, then $a_0(H)+a_1(H)+a_2(H)\leq \frac12$ and thus $a_3(H)\geq -\frac12$.
This shows that $H$ is also in state-$1$, so by Lemma \ref{state1} we are done. 
Hence from now on, we may assume that $a_1(H)=\frac52$. 
Note that we have $a_2(H)+a_3(H)=0$ and $x:=a_2(H)\in \{-\frac12, \frac12, \frac32\}$.

If there exists an $H(1,1)$-edge $e$ in $H$, then $H_1= H - e$ satisfies $\mathbf{a}(H_1) = \mathbf{a}(H) + (2,-2,0,0)=(-\frac12,\frac12,x,-x)$. 
If there exists some $S_{(2;1,1)}(H)$,
then the subgraph $H_2$ obtained from $H$ by deleting all edges of this $S_{(2;1,1)}(H)$ satisfies that $\mathbf{a}(H_2) = \mathbf{a}(H) + (3,-2,-1,0) = (\frac12,\frac12,x-1,1-x)$.
Note that each of $H_1$ and $H_2$ is proper whenever $x\in \{-\frac12, \frac12, \frac32\}$.
So we may claim that $p_{11}(H)=0$ and there is no $S_{(2;1,1)}(H)$.

Suppose $x\in \{-\frac12, \frac12\}$. 
If there is an $H(1,3)$-edge $f$ in $H$, then $H_3= H - f$ satisfies $\mathbf{a}(H_3) = \mathbf{a}(H) + (1,-1,1,-1)=(-\frac32,\frac32,x+1,-x-1)$ which is proper.
So we may assume $p_{13}(H) = 0$.
Since $p_{11}(H)=0$ and $|V_1^H| > |V_2^H|$,  by Lemma~\ref{ext} there exists some $S_{(2;1,1)}(H)$, 
which is a contradiction.
It remains to consider the case $x=\frac32$.
That is, $\mathbf{a}(H)=(-\frac52,\frac52,\frac32,-\frac32)$.

Suppose that there exist an $H(1,2)$-edge $e$ and an $H(1,3)$-edge $f$ in $H$.
Clearly, $e$ and $f$ are vertex-disjoint. 
Let $H_4 = H - e - f\subseteq G$. 
Then we have 
$\mathbf{a}(H_4) = \mathbf{a}(H) + (2,-1,0,-1)= (-\frac12,\frac32,\frac32,-\frac52)$, implying that $G\backslash H_4$ is in state-$1$.
By Lemma~\ref{state1}, $G$ has a proper subgraph.  

Hence, we may assume that either $p_{12}(H)=0$ or $p_{13}(H)=0$.
In the former case, we have $p_{11}(H)=p_{12}(H)=0$, and since $|V_1^H| > |V_3^H|$, 
by Lemma~\ref{ext} there exists some $S_{(3;1,1)}(H)$.
Let $H_5$ be obtained from $H$ by deleting all edges of this $S_{(3;1,1)}(H)$. 
Then $\mathbf{a}(H_5) = \mathbf{a}(H) + (2,-1,0,-1) = (-\frac12,\frac32,\frac32,-\frac52)$, 
implying that $G\backslash H_5$ is in state-$1$.
Using Lemma \ref{state1} again, there exists a proper subgraph of $G$. 
So we may assume that the latter case $p_{13}(H)=0$ holds.
However, since $p_{11}(H)=0$ and $|V_1^H| > |V_2^H|$, 
by Lemma~\ref{ext} there exists some $S_{(2;1,1)}(H)$,
a contradiction to the above claim.

Similar to the previous lemma, this proof also can be implemented in $O(1)$ time,
if $\mathcal{P}_{ij}(H)$, $\mathcal{Q}_{ij}(H)$, $\mathcal{S}_i(H)$ and $\mathcal{S}_i(G\backslash H)$ for all $0\leq i,j\leq 3$ are provided.
\end{proof}

\subsection{Completing the proof of Theorem~\ref{thm2}}

Before we can conclude the proof of Theorem~\ref{thm2},
we need an additional lemma that addresses the inductive step depicted in Figure~\ref{fig:induction}.

\begin{lemma}\label{op1}
Let $G, G'\in \mathcal{G}$ satisfy that $G\hookrightarrow G'$.
If $G$ has a subgraph $H$ in state-$0$, then $G'$ has a proper subgraph.
In addition, if $\mathcal{P}_{ij}(H)$, $\mathcal{Q}_{ij}(H)$, $\mathcal{S}_i(H)$ and $\mathcal{S}_i(G\backslash H)$ for all $0\leq i,j\leq 3$ are provided,
then one can construct a proper subgraph of $G'$ using $O(1)$ time.
\end{lemma}

\begin{proof}
There are two types of the operation $G\hookrightarrow G'$,
and we distinguish between them. 

Suppose $G\hookrightarrow G'$ is of Type (I). 
Let $xy\in E(G)$ and $u,v\in V(G')\backslash V(G)$ be
as indicated on the left-hand side of Figure~\ref{Fig:operations}. 
Let $H$ be a spanning subgraph of $G$ in state-$0$.
So 
\begin{equation}\label{equ:state-0}
	-2\leq a_0(H), a_3(H)\leq \frac12 \mbox{ and } -\frac12\leq a_1(H), a_2(H)\leq 2.
\end{equation}
By Proposition~\ref{prop:states} (2), $G\backslash H$ is also in state-$0$. 
By symmetry between $H$ and  $G\backslash H$, we may assume that $xy\in E(H)$. 
Let $H'$ be obtained from $H$ by removing the edge $xy$, adding vertices $u$ and $v$, and adding the edges $xu, vy$ and two parallel edges between $u$ and $v$.
It is evident to see that $H'\subseteq G'$ with $V_i^{H'} = V_i^H$ for $i=0,1,2$ and $V_3^{H'} = V_3^{H}\cup\{u,v\}$.
Since $|V(G')|=|V(G)|+2$,
we have $a_i(H') = a_i(H) - \frac12$ for $i=0,1,2$ and $a_3(H') = a_3(H) + \frac32$.
Then $H'$ is a proper subgraph of $G'$, unless that $a_0(H) = -2$. 
In the latter case, we see $a_0(H') = -\frac52$, $a_1(H'),a_2(H')\in [-1,\frac32]$ and $a_3(H')\in [-\frac12,2]$.
By Proposition~\ref{prop:states} (1), all $a_i(H') \in \mathbb{Z} + \frac{1}{2}$.
This shows that $a_0(H') = -\frac52$ and $a_1(H'),a_2(H'), a_3(H')\in [-\frac12,\frac32]$, i.e., $H'$ is in state-$1$. 
By Lemma \ref{state1}, $G'$ has a proper subgraph. 
For the algorithmic aspect, 
suppose that $\mathcal{P}_{ij}(H)$, $\mathcal{Q}_{ij}(H)$, $\mathcal{S}_i(H)$ and $\mathcal{S}_i(G\backslash H)$ for all $0\leq i,j\leq 3$ are provided.
Then it takes $O(1)$ time to update and derive $\mathcal{P}_{ij}(H')$, $\mathcal{Q}_{ij}(H')$, $\mathcal{S}_i(H')$ and $\mathcal{S}_i(G'\backslash H')$ for all $0\leq i,j\leq 3$. 
Using Lemma \ref{state1} again, one can construct a proper subgraph of $G'$ using $O(1)$ time.

Now we may assume that $G\hookrightarrow G'$ is of Type (II). 
Let $xy,zw\in E(G)$ and $u,v\in V(G')\backslash V(G)$ be as indicated on the right-hand side of Figure~\ref{Fig:operations}. 
Let $H$ be a spanning subgraph of $G$ in state-$0$. 
So $\mathbf{a}(H)$ satisfies \eqref{equ:state-0}.
By symmetry between $H$ and  $G\backslash H$, we may assume that $xy\in E(H)$.  
Suppose that $zw\in E(H)$.
Let $H'$ be obtained from $H$ by removing the edges $xy$ and $zw$, adding vertices $u$ and $v$, and adding the edges $xu,uy,zv,vw$ and $uv$.
Then $H'$ is a subgraph of $G'$ with $V_i^{H'}= V_i^H$ for $i=0,1,2$ and $V_3^{H'} = V_3^H\cup\{u,v\}$.
Following the exactly same proof as in the last paragraph, we can derive a proper subgraph in $G'$.

Therefore, we may assume that $xy\in E(H)$ and $zw\notin E(H)$.
Let $H^*$ be obtained from $H$ by removing the edge $xy$, adding vertices $u$ and $v$, and adding the edges $xu,uy$ and $uv$. 
Then we see $V_i^{H^*}=V_i^H$ for $i\in \{0,2\}$, $V_1^{H^*}=V_1^H\cup \{v\}$ and $V_3^{H^*}=V_3^{H}\cup \{u\}$,
implying that $$\mathbf{a}(H^*) = \mathbf{a}(H) + \left(-\frac12,\frac12,-\frac12,\frac12\right).$$
Since $\mathbf{a}(H)$ satisfies \eqref{equ:state-0},
we see that $H^*$ is a proper subgraph of $G'$, unless one of the two possibilities $a_0(H^*)=-\frac52$ and $a_1(H^*)=\frac52$ occurs. 
Now suppose that $a_0(H^*)=-\frac52$.
Using Proposition~\ref{prop:states} (1), we can derive that
$a_1(H^*)\in [\frac12, \frac52]$, $a_2(H^*)\in [-\frac12, \frac32]$ and $a_3(H^*)\in [-\frac32, \frac12]$. 
That is, $H^*$ is in state-$2$. 
By Lemma~\ref{state2}, one can find a proper subgraph in $G'$.
Therefore, we may assume that $a_0(H^*)\geq -2$ and $a_1(H^*)=\frac52$. 
Using Proposition~\ref{prop:states} (1) again,
we have $a_0(H^*)\in [-\frac32, -\frac12]$, $a_2(H^*)\in [-\frac12, \frac32]$ and $a_3(H^*)\in [-\frac32, \frac12]$. 
This shows that $a_0(H^*)+a_1(H^*)+a_3(H^*)\geq -\frac12$ and thus $a_2(H^*)\in [-\frac12,\frac12]$.
Note that $uv$ is an $H^*(1,3)$-edge in $H^*$.
Then $H^{\star}=H^*-uv$ is a subgraph of $G'$ satisfying $$\mathbf{a}(H^{\star}) = \mathbf{a}(H^*) +(1,-1,1,-1).$$ 
This shows that
$a_0(H^{\star})\in [-\frac12,\frac12]$, $a_1(H^{\star}) = \frac32$, $a_2(H^{\star})\in [\frac12,\frac32]$ and $a_3(H^{\star})\in [-\frac52, -\frac12]$. 
It is key to observe that $G'\backslash H^\star$ is proper or it is in state-$1$.
Now, applying Lemma~\ref{state1} once again, the proof of this lemma is completed.
In the case of Type (II), we omit the ``moreover" part as it can be derived analogously to the proof for Type (I).
\end{proof}

Finally, we can establish Theorem~\ref{thm2} by proving the following slightly stronger statement.

\begin{theorem}\label{thm2:stronger}
Let $G$ be any cubic multigraph.
Then there exists a spanning subgraph $H$ of $G$ in state-$0$, that is, $a_0(H), a_3(H)\in [-2, \frac12]$ and $a_1(H), a_2(H)\in [-\frac12,2].$
Moreover, there exists a linear time algorithm to find such a spanning subgraph $H$.
\end{theorem}

\begin{proof}
Let $G$ be any cubic multigraph on $2n$ vertices. 
First, we point out that one can find all connected components of $G$ and the multiplicities of all edges of $G$ in linear time $O(n)$.\footnote{To do so, one can employ a Depth-First Search Tree algorithm on $G$.}

To finish the proof, it is enough to prove the following statement by induction on $n$:

\medskip 

\noindent {\bf ($\star$)} There exists an absolute constant $C>0$ such that
provided all connected components of a cubic multigraph $G$ on $2n$ vertices and the multiplicities of all edges of $G$, 
one can use at most $C\cdot n$ time to construct a spanning subgraph $H$ of $G$ in state-$0$ as well as $\mathcal{P}_{ij}(H)$, $\mathcal{Q}_{ij}(H)$, $\mathcal{S}_i(H)$ and $\mathcal{S}_i(G\backslash H)$ for all possible $0\leq i,j\leq 3$.

\medskip

We choose a sufficiently large constant $C>0$ that surpasses the $O(1)$ constant terms in Lemmas~\ref{ctr}, \ref{lem:state-0} and \ref{op1}.
We first prove the statement ($\star$) for $nK_2^3$, where $n\geq 1$ is arbitrary. This is sufficient to establish the base case.
Let $n = 4k + r$ where $0\leq r\leq 3$.
Take $H'$ as the disjoint union of $kK_2^0$, $kK_2^1$, $kK_2^2$ and $k K_2^3$.
Then $H'$ is a proper subgraph of $4k K_2^3$ with $a_i(H) = 0$ for all $0\leq i\leq 3$. 
If $r=0$, let $H=H'$; 
otherwise $1\leq r\leq 3$,
let $H$ be the disjoint union of $H'$ and one copy of each of the first $r$ graphs in the sequence $K_2^1, K_2^2, K_2^0, K_2^3$.
In the latter case, for each $i$ among the first $r$ indices in the sequence $1, 2,0, 3$, we have $a_i(H) = 2 - \frac{r}{2}$, 
and for other indices $j$, we have  
$a_j(H) = -\frac{r}2$.
So $H$ is a desired spanning subgraph of $nK_2^3$ in state-$0$.
It is also straightforward to constuct $\mathcal{P}_{ij}(H)$, $\mathcal{Q}_{ij}(H)$ and $\mathcal{S}_i(H)=\emptyset=\mathcal{S}_i(nK_2^3\backslash H)$ for all $0\leq i,j\leq 3$ (say with at most $C\cdot n$ time).

Now assume ($\star$) holds for any cubic multigraph with less than $2n$ vertices. 
Let $G\in \mathcal{G}$ be a cubic multigraph on $2n$ vertices, with given connected components and multiplicities of edges. 
We may assume that $G$ has a connected component $F\in \mathcal{G}$ with at least four vertices. 
By Lemma~\ref{ctr}, one can use at most $C/3$ time
to construct $F^*\in \mathcal{G}$ (and also update the multiplicities of all edges of $F^*$)
such that $F^*\hookrightarrow F$.
Then $G^*\hookrightarrow G$, where $G^*\in \mathcal{G}$ denotes the disjoint union of $G-V(F)$ and $F^*$.
It is worth noting that $|V(G^*)|=|V(G)|-2=2(n-1)$, and all connected components of $G^*$ and the multiplicities of all edges of $G^*$ are inherited from those of $G$ (plus $F^*$).
By applying induction to $G^*$, one can use at most $C\cdot (n-1)$ time to construct a spanning subgraph $H^*$ of $G^*$ in state-$0$ as well as $\mathcal{P}_{ij}(H^*)$, $\mathcal{Q}_{ij}(H^*)$, $\mathcal{S}_i(H^*)$ and $\mathcal{S}_i(G^*\backslash H^*)$ for all possible $0\leq i,j\leq 3$.
By Lemma~\ref{op1}, 
one can use at most $C/3$ time to construct a proper subgraph $H_1$ of $G$
and update $\mathcal{P}_{ij}(H_1)$, $\mathcal{Q}_{ij}(H_1)$, $\mathcal{S}_i(H_1)$ and $\mathcal{S}_i(G\backslash H_1)$ for all $0\leq i,j\leq 3$.
Then using Lemma~\ref{lem:state-0}, 
one can further take at most $C/3$ time to construct a subgraph $H_2$ of $G$ in state-$0$ and update
$\mathcal{P}_{ij}(H_2)$, $\mathcal{Q}_{ij}(H_2)$, $\mathcal{S}_i(H_2)$ and $\mathcal{S}_i(G\backslash H_2)$ for all $0\leq i,j\leq 3$.
Note that the construction of the desired $H_2$ and the corresponding families requires at most $C \cdot n$ time in total. 
We have completed the proof of the statement ($\star$).
\end{proof}

\section{Concluding remarks}\label{Sec:remarks}
In this paper, we address Conjecture~\ref{conj} proposed by Alon and Wei.
We prove two main results: a upper bound $2d^2$, which is the first bound that is independent of the number $n$ of vertices, and a confirmation of the case $d=3$. Both results extend to multigraphs and yield efficient algorithms.   
While we are aware that a similar argument, albeit potentially more extensive, can be employed to prove the conjecture for small values of $d$, it is evident that additional innovative ideas will be required to resolve the conjecture in its general form.

While investigating related properties of multigraphs, 
we find it intriguing to ask the following strengthening of Conjecture~\ref{conj}.

\begin{conjecture}\label{new_conjecture}
There exists an absolute constant $c>0$ such that for any $n, d\geq 2$ and any $n$-vertex $d$-regular multigraph $G$ which does not contain $K_2^d$, there exists a spanning subgraph $H$ of $G$ satisfying
\begin{equation}\label{equ_of_new_conjecture}
	\left|m(H,k) - \frac{n}{d+1}\right|\leq 1 + \frac{c}{n} \mbox{~ for all $0\leq k \leq d$. }
\end{equation}
\end{conjecture}

We exclude the multigraph $K_2^d$ in this conjecture, because otherwise the disjoint union $nK_2^d$ would violate \eqref{equ_of_new_conjecture} for some appreciate choices of $d,n$.
There exist examples showing that the bound of \eqref{equ_of_new_conjecture} is tight up to the constant $c$.
The first example is the complete bipartite graphs $K_{d,d}$ for odd $d$,
for which we can derive from the footnote of the second page that 
\[
\min_{H\subseteq K_{d,d}}\|\mathbf{a}(H)\|_{\infty} = 3 - \frac{2d}{d+1}= 1 + \frac{4}{n+2}.
\]
Another example is the multigraph $K^2_{d/2,d/2}$ obtained from $K_{d/2,d/2}$ by replacing every edge with a pair of parallel edges for $d$ satisfying $d\equiv 2 \mod 4$.
Using a result of Gy\'{a}rf\'{a}s \cite{Gya1998}, 
any spanning subgraph $H$ of $K^2_{d/2,d/2}$ has $m(H,k)\geq 2$ for some $0\leq k \leq d$, implying that
\[\min_{H\subseteq K^2_{d/2,d/2}}\|\mathbf{a}(H)\|_{\infty}= 2 - \frac{d}{d+1} = 1 + \frac{1}{n+1}.\]
Our proof of Theorem~\ref{thm2} may potentially be extended to cover the case $d=3$ of Conjecture~\ref{new_conjecture}.

Another motivation for considering Conjecture~\ref{new_conjecture} arises from its connection to the Faudree-Lehel conjecture which we restate here.

\begin{conjecture}[Faudree and Lehel \cite{FL1988}]\label{conj:irregularity_strength}
There exists an absolute constant $c> 0$ such that for any $n > d \geq 2$ and any $n$-vertex $d$-regular graph $G$, $$s(G) \leq \frac{n}{d} + c.$$   
\end{conjecture}

We point out that a confirmation of Conjecture~\ref{new_conjecture} would imply Conjecture~\ref{conj:irregularity_strength}.

\begin{proposition}\label{new_conj_to_irregularity_strength}
If Conjecture \ref{new_conjecture} holds, then Conjecture \ref{conj:irregularity_strength} holds.
\end{proposition}

\begin{proof}
Let $n>d\geq 2$ and $G$ be any $n$-vertex $d$-regular graph. 
For integers $s\geq 1$, let $G^s$ be the multigraph generated from $G$ by replacing every edge $e = uv\in E(G)$ by $s$ parallel edges $e_1, \ldots, e_s$ with endpoints $u$ and $v$. 
Let $E_e = \{e_1,\ldots,e_s\}$ for each $e\in E(G)$.

Note that $G^s$ is $sd$-regular.
We first claim that if $G^s$ has a spanning subgraph $H$ such that $m(H,k) \in \{0,1\}$ for all $0\leq k\leq sd$, 
then $s(G)\leq s+1$.
To see this, let $f_H$ be the $(s+1)$-edge-weighting of $G$ satisfying $f_H(e):= |E_e\cap H| + 1\in \{1,2,\ldots, s+1\}$. 
Then for all $0\leq k\leq sd$, we have $|\{v:\sum_{e:v\sim e}f_H(e)=k+d\}|= m(H,k)\leq 1$. 
Equivalently, this says that $\sum_{e: v\sim e} f_H(e)$ are distinct for all $v\in V(G)$.

Assume that Conjecture \ref{new_conjecture} holds. 
Since $d\geq 2$, it is evident that $G^s$ does not contain $K_2^{sd}$.
Then there exists an absolute constant $c>0$ such that for any $s\geq 1$,
$G^s$ contains a spanning subgraph $H$ with 
$$\left|m(H,k) - \frac{n}{sd+1}\right|\leq 1 + \frac{c}{n}$$ 
for all $0\leq k \leq sd$. 
Now choosing $s=\frac{n}{d}+c$, so $sd\geq n+2c$.
If $n\geq 2c+1$, then for all $k$,
\[
m(H,k)\leq 1+\frac{n}{sd+1}+\frac{c}{n}\leq 1+\frac{n}{n+2c+1}+\frac{2c}{n+2c+1}<2.
\]
Hence, $m(H,k)\in \{0,1\}$ holds for all $0\leq k\leq sd$.
Using the above claim,
we see that $s(G)\leq \frac{n}{d}+c+1$ for any $n$-vertex $d$-regular graph $G$ with $n\geq 2c+1$. 
In the case when $n<2c+1$, 
there are finite many $d$-regular graphs for all $d\geq 2$.
Since the irregularity strength $s(G)$ exists for such $G$ due to Faudree-Lehel \cite{FL1988},
we may take a sufficiently large constant $C\gg c$ such that 
$s(G)\leq \frac{n}{d}+C$ for any $n$-vertex $d$-regular graph $G$, 
finishing the proof.
\end{proof}

Alon and Wei \cite{ALO2022} also proposed the following conjecture, which can be viewed as a one-sided generalization of Conjecture~\ref{conj}.

\begin{conjecture}[Alon and Wei \cite{ALO2022}]\label{conj2}
Let $G$ be an $n$-vertex graph with minimum degree $\delta$ and maximum degree $\Delta$. Then there exists a spanning subgraph $H$ in $G$ such that $m(H,k) \leq  \frac{n}{\delta+1} + 2$ holds for every $k\geq 0$.
\end{conjecture}

Despite their similar spirit, it is important to highlight a significant distinction between Conjecture~\ref{conj} and Conjecture~\ref{conj2}. In Conjecture~\ref{conj2}, the symmetry between $H$ and $G\backslash H$ no longer holds, whereas this symmetry is crucial in our proofs concerning Conjecture~\ref{conj}.
Alon and Wei have addressed several other open problems in \cite{ALO2022}.
For further details and in-depth discussions, we would like to direct interested readers to refer to \cite{ALO2022}.

\end{document}

\begin{figure}
\centering
\begin{minipage}{.4\textwidth}
\centering
\begin{tikzpicture}[thin][node distance=1cm,on grid]
	\node[circle,inner sep=0.5mm,draw=black,fill=black](x1)at(180:2) [label=above:$x$] {};
	\node[circle,inner sep=0.5mm,draw=black,fill=black](x2)at(0:2) [label=above:$y$] {};
	\node[circle,inner sep=0.5mm,draw=black,fill=black](x3)at(0:0.8) [label=below:$v$] {};
	\node[circle,inner sep=0.5mm,draw=black,fill=black](x4)at(180:0.8) [label=below:$u$] {};
	\draw (0,-1.2) node{};
	\draw (0,1.4) node{};
	\draw[draw,thick,color=blue] (0.8,0) arc (37:143:1);
	\draw[draw,thick,color=blue] (0.8,0) arc (-37:-143:1);
	\draw[thick,color = blue](x1)to(x4);
	\draw[thick,color = blue](x2)to(x3);
	\draw[draw,color=orange] (2.0,0) arc (37:143:5/2);
\end{tikzpicture}	
\label{Fig:operation 1}
\subcaption{Type (\uppercase\expandafter{\romannumeral 1})}
\end{minipage}
\hspace{1.2cm}
\begin{minipage}{.4\textwidth}
\centering
\begin{tikzpicture}[thin][node distance=1cm,on grid]
	\node[circle,inner sep=0.5mm,draw=black,fill=black](x1)at(-2,0.8) [label=above:$x$] {};
	\node[circle,inner sep=0.5mm,draw=black,fill=black](x2)at(2,0.8) [label=above:$y$] {};
	\node[circle,inner sep=0.5mm,draw=black,fill=black](x3)at(-2,-0.8) [label=below:$z$] {};
	\node[circle,inner sep=0.5mm,draw=black,fill=black](x4)at(2,-0.8) [label=below:$w$] {};
	\node[circle,inner sep=0.5mm,draw=black,fill=black](x5)at(0,0.8) [label=above:$u$] {};
	\node[circle,inner sep=0.5mm,draw=black,fill=black](x6)at(0,-0.8) [label=below:$v$] {};
	\draw[thick,color = blue](x1)to(x5);
	\draw[thick,color = blue](x2)to(x5);
	\draw[thick,color = green](x3)to(x6);
	\draw[thick,color = green](x4)to(x6);
	\draw[thick,color = blue](x5)to(x6);
	\draw[draw,color=orange] (2.0,0.8) arc (63.45:116.55:4.472);
	\draw[draw,color=orange] (-2.0,-0.8) arc (243.45:296.55:4.472);
\end{tikzpicture}
\label{Fig:operation 2}
\subcaption{Type (\uppercase\expandafter{\romannumeral 2})}
\end{minipage}
\caption{Key steps in the proof of Lemma~\ref{op1}}
\end{figure}

